\documentclass[reqno,a4paper]{amsart}
\usepackage{amssymb,amsmath,amsfonts,mathptmx,cite,mathrsfs,url}
\usepackage{comment}
\usepackage{enumerate}
\usepackage{tikz}
\usepackage[breaklinks,pagebackref,hyperindex=false]{hyperref}

\newcommand{\operator}{\mathsf}
\DeclareMathOperator{\content}{\operator{con}}

\DeclareMathOperator{\var}{\operator{var}}

\newtheorem{theorem}{Theorem}[section]
\newtheorem{proposition}[theorem]{Proposition}
\newtheorem{lemma}[theorem]{Lemma}
\newtheorem{corollary}[theorem]{Corollary}

\theoremstyle{remark}
\newtheorem{remark}[theorem]{Remark}

\numberwithin{equation}{section}



\newcommand{\Jtrivial}{$J$-trivial}

\newcommand{\monoid}{ }
\newcommand{\monAzero}{\monoid{A_0}}
\newcommand{\monBzero}{\monoid{B_0}}
\newcommand{\monE}{\monoid{E}} \newcommand{\monEdual}{\overleftarrow{\monE}}
\newcommand{\monFsub}{\monoid{F}_1}
\newcommand{\monHsub}{\monoid{H_3}} \newcommand{\monHsubdual}{\overleftarrow{\,\monHsub}}
\newcommand{\monK}{\monoid{K}}
\newcommand{\monM}{\monoid{M}} 
\newcommand{\monQ}{\monoid{Q}}
\newcommand{\monRq}{\monoid{Rq}}

\newcommand{\variety}{\mathbf}
\newcommand{\bfAzero}{\variety{\monAzero}}
\newcommand{\bfBzero}{\variety{\monBzero}}
\newcommand{\bfE}{\variety{E}} \newcommand{\bfEdual}{\overleftarrow{\bfE}}
\newcommand{\bfF}{\variety{F}} \newcommand{\bfFdual}{\overleftarrow{\bfF}}
\newcommand{\bfFsub}{\variety{F}_1}
\newcommand{\bfH}{\variety{H}} \newcommand{\bfHdual}{\overleftarrow{\bfH}}
\newcommand{\bfHsub}{\variety{H}_3} \newcommand{\bfHsubdual}{\overleftarrow{\bfH}\!{}_3}
\newcommand{\bfI}{\variety{I}} \newcommand{\bfIdual}{\overleftarrow{\bfI}}
\newcommand{\bfIsub}{\variety{T}} 
\newcommand{\bfJ}{\variety{J}}
\newcommand{\bfK}{\variety{K}} \newcommand{\bfKdual}{\overleftarrow{\bfK}}
\newcommand{\bfKsub}{\variety{N}}
\newcommand{\bfL}{\variety{L}}

\newcommand{\bfO}{\variety{O}}
\newcommand{\bfP}{\variety{P}} \newcommand{\bfPdual}{\overleftarrow{\bfP}}
\newcommand{\bfPsub}{\variety{W}}
\newcommand{\bfQ}{\variety{Q}}
 
\newcommand{\bfRq}{\variety{Rq}}
\newcommand{\bfU}{\variety{U}}
\newcommand{\bfV}{\variety{V}} 

\newcommand{\bfJackson}{\variety{Y}}
\newcommand{\bfZL}{\variety{Z}}
\newcommand{\bftrivial}{\variety{0}}

\newcommand{\word}{\mathbf}
\newcommand{\bfa}{\word{a}}
\newcommand{\bfb}{\word{b}}
\newcommand{\bfc}{\word{c}}

\newcommand{\bfh}{\word{h}}
\newcommand{\bfp}{\word{p}}
\newcommand{\bfq}{\word{q}}
\newcommand{\bfu}{\word{u}}  
\newcommand{\bfv}{\word{v}}  
\newcommand{\bfw}{\word{w}}  
\newcommand{\bfx}{\word{x}}

\newcommand{\scrA}{\mathscr{A}}
\newcommand{\scrW}{\mathscr{W}}
\newcommand{\scrX}{\mathscr{X}}
\newcommand{\frakL}{\mathfrak{L}}

\newcommand{\class}{\mathbb}
\newcommand{\bbAcen}{\class{A}^\mathrm{cen}}
\newcommand{\bbAcom}{\class{A}^\mathrm{com}}
\newcommand{\bbJ}{\class{J}}
\newcommand{\bbM}{\class{M}}

\newcommand{\etal}{\textit{et~al}}
\newcommand{\aper}{\text{\textcircled{a}}} \newcommand{\ecom}{\text{\textcircled{c}}}
\newcommand{\xxyx}{\blacktriangleleft} \newcommand{\xyxx}{\blacktriangleright}
\newcommand{\Osub}{\circledR}

\newcommand{\dt}{0.12}

\hypersetup{
  colorlinks = true, 
  urlcolor   = cyan, 
  linkcolor  = blue, 
  citecolor  = red 
}

\allowdisplaybreaks

\begin{document}

\title{Characterization of Cross varieties of $J$-trivial monoids}
\thanks{S.V.\ Gusev was supported by the Russian Science Foundation (no.~25-71-00005); W.T.\ Zhang was partially supported by the National Natural Science Foundation of China (nos.~12271224 and 12571018) and the Fundamental Research Funds for the Central Universities (no.~lzujbky-2023-ey06).} 

\author[S.V. Gusev]{Sergey V. Gusev} \address{Institute of Natural Sciences and Mathematics, Ural Federal University, 620000 Ekaterinburg, Russia} \email{sergey.gusev@urfu.ru}

\author[E.W.H. Lee]{Edmond W. H. Lee} \address{Department of Mathematics, Nova Southeastern University, Fort Lauderdale, Florida 33328, USA} \email{edmond.lee@nova.edu}

\author[W.T. Zhang]{Wen Ting Zhang} \address{School of Mathematics and Statistics, Lanzhou University, Lanzhou, Gansu 730000, P.R. China} \email{zhangwt@lzu.edu.cn}



\begin{abstract}
A finitely based, finitely generated variety with finitely many subvarieties is a \textit{Cross} variety.
In the present article, it is shown that a variety of $J$-trivial monoids is Cross if and only if it excludes as subvarieties a certain list of 14 almost Cross varieties.
Consequently, the list of 14 varieties exhausts all almost Cross varieties of $J$-trivial monoids. 
\end{abstract}



\maketitle

\tableofcontents

\section{Introduction} \label{sec: introduction}

A nonempty class of algebras of a fixed type---such as groups, rings, and semigroups---is a \textit{variety} if it is closed under the formation of subalgebras, homomorphic images, and arbitrary direct products.
By Birkhoff's theorem~\cite{Bir35}, varieties are precisely equationally defined classes of algebras and hence can be investigated by both semantic and syntactic methods.
Finiteness conditions relevant to these methods naturally receive much attention: a variety is \textit{finitely based} if its equational theory is finitely axiomatizable, \textit{finitely generated} if it is generated by a finite algebra, and \textit{small} if it has a finite lattice of subvarieties.
These three conditions are independent in the sense that a variety that satisfies any two conditions need not satisfy the third~\cite{SapMV91}.
Following Higman~\cite{Hig60}, any variety that satisfies all three conditions is said to be \textit{Cross}.

Finite members from well-studied classes of algebras, such as groups~\cite{OatPow64}, associative rings~\cite{Kru73,Lvo73}, Lie rings~\cite{BahOls75}, and lattices~\cite{McK70}, generate Cross varieties.
But this result does not hold for more general algebras, a striking counterexample being the Brandt monoid \[ B = \left\{ \begin{bmatrix}0&0\\0&0\end{bmatrix}, \begin{bmatrix}1&0\\0&0\end{bmatrix}, \begin{bmatrix}0&1\\0&0\end{bmatrix}, \begin{bmatrix}0&0\\1&0\end{bmatrix}, \begin{bmatrix}0&0\\0&1\end{bmatrix}, \begin{bmatrix}1&0\\0&1\end{bmatrix} \right\} \] under usual matrix multiplication; not only is the variety generated by~$B$ non-finitely based~\cite{SapMV87}, its subvarieties form an uncountable lattice~\cite{JacLee18} that embeds every finite lattice~\cite{Gus24}.
For any class of algebras that contains non-Cross varieties, one approach to characterizing Cross varieties is to identify its minimal non-Cross varieties, which are commonly called \textit{almost Cross} varieties. 
It follows from Zorn's lemma that the exclusion of almost Cross subvarieties is not only necessary but also sufficient for a variety to be Cross.

The present article is concerned with varieties of monoids.
Classical examples of almost Cross varieties include the varieties \[ \mathbf{Com} = \var\{xy \approx yx\} \ \text{ and } \ \mathbf{Idem} = \var\{ x^2 \approx x \} \] of commutative monoids~\cite{Hea68} and idempotent monoids~\cite{Wis86}, respectively, which are both non-finitely generated.
Since a variety of groups is Cross if and only if it is finitely generated~\cite{OatPow64}, describing all almost Cross varieties of groups may not seem challenging; but the task is actually closer to being impossible, as there are already uncountably many of them covering just the variety \[ \mathbf{Z}_p = \var\{ x^p \approx 1,\, xy \approx yx\} \] generated by the cyclic group of each sufficiently large prime order~$p$~\cite{Koz12}.
Consequently, a characterization of all Cross varieties of monoids is infeasible, and it is logical to focus on monoids without nontrivial subgroups; such a monoid satisfies the identity $x^{n+1} \approx x^n$ for some $n \geq 0$ and is said to be \textit{aperiodic}.

Some classes in which almost Cross varieties have been completely described are the class~$\bbAcom$ of aperiodic monoids with commuting idempotents and its subclass~$\bbAcen$ of aperiodic monoids with central idempotents.
In 2005, Jackson~\cite{Jac05} published the first two examples of almost Cross subvarieties of~$\bbAcen$; these varieties, denoted by $\bfJackson_1$ and $\bfJackson_2$ (see Section~\ref{subsec: var Y1 Y2}), inspired the investigation of Cross subvarieties of~$\bbAcen$~\cite{Lee11}.
In 2013, a third almost Cross subvariety of~$\bbAcen$, denoted by~$\bfL$ (see Section~\ref{subsec: var L P}), was published; not only are~$\bfL, \bfJackson_1,\bfJackson_2$ the only almost Cross subvarieties of~$\bbAcen$~\cite{Lee13}, they are also the only almost Cross subvarieties of~$\bbAcom$ satisfying the identity $x^2 h x \approx xhx^2$ \cite{Lee14rm}.
In the years that followed, further examples of almost Cross subvarieties of $\bbAcom$ were found \cite{Gus20,Gus21,GusVer18}.
A description of all almost Cross subvarieties of~$\bbAcom$ was eventually completed by Gusev~\cite{Gus25ijac}.
Specifically, there are precisely nine almost Cross subvarieties of~$\bbAcom$: the subvarieties $\bfL, \bfJackson_1, \bfJackson_2$ of~$\bbAcen$ and the subvarieties $\bfF,\bfI,\bfP$ of $\bbAcom$ together with their duals $\bfFdual, \bfIdual,\bfPdual$ (see Sections~\ref{subsec: var F}, \ref{subsec: var L P}, and~\ref{subsec: var I}).

An obvious next step in the investigation is to extend the description of almost Cross varieties to other classes of monoids; one natural candidate is the class~$\bbJ$ of {\Jtrivial} monoids since it already contains all nine almost Cross subvarieties of~$\bbAcom$.
Recall that a monoid~$\monM$ is \textit{\Jtrivial} if distinct elements in~$\monM$ generate distinct principal two-sided ideals, that is, $MaM \neq MbM$ for all distinct $a,b \in \monM$.
The fundamental importance of {\Jtrivial} monoids lies in automata theory and formal language theory; in particular, finite {\Jtrivial} monoids correspond precisely to the class of recognizable languages known as the piecewise testable events~\cite{Sim75}. 
Presently, 12 subvarieties of~$\bbJ$ are known to be almost Cross: the aforementioned nine subvarieties of~$\bbAcom$, a certain variety~$\bfK$ from Gusev and Sapir~\cite{GusSap22} with its dual~$\bfKdual$, and a certain variety~$\bfZL$ from Zhang and Luo~\cite{ZhaLuo19} (see Sections~\ref{subsec: var K} and~\ref{subsec: var ZL}).

The objective of this article is to characterize all Cross subvarieties of~$\bbJ$ by providing a complete description of its almost Cross subvarieties.
Some background information and results are first established in Section~\ref{sec: prelim}.
Then almost Cross subvarieties of~$\bbJ$ and related information required to establish the main result are given in Sections~\ref{sec: almost nfg} and~\ref{sec: almost fg}; in particular, a new almost Cross variety~$\bfH$ and its dual~$\bfHdual$ are exhibited in Section~\ref{subsec: var H}, bringing the total number of known almost Cross subvarieties of~$\bbJ$ to~14.
In Section~\ref{sec: characterization}, it is shown that a subvariety of~$\bbJ$ is Cross if and only if it excludes these 14 almost Cross varieties, thus completing the description of all almost Cross subvarieties of~$\bbJ$.

\section{Preliminaries} \label{sec: prelim}

Acquaintance with rudiments of universal algebra is assumed of the reader; refer to Almeida~\cite{Alm94} and Burris and Sankappanavar~\cite{BurSan81} for any undefined concepts.
For more information on varieties of monoids, see Gusev \etal.\ \cite{GusLeeVer22} and Lee \cite[Section~1.6]{Lee23}.

\subsection{Words and identities}

Let~$\scrA$ be a countably infinite alphabet.
For any subset~$\scrX$ of~$\scrA$, let~$\scrX^+$ and~$\scrX^*$ denote the free semigroup and free monoid over~$\scrX$, respectively. 
Elements of~$\scrA$ and~$\scrA^*$ are called \textit{variables} and \textit{words}, respectively. 
The empty word, written as~$1$, is the identity element of the monoid~$\scrA^*$; in other words, $\scrA^* = \scrA^+ \cup \{ 1 \}$.
The \textit{content} of a word~$\bfw$, denoted by $\content(\bfw)$, is the set of variables occurring in~$\bfw$.
The number of times a variable~$x$ occurs in~$\bfw$ is denoted by~$|\bfw|_x$.
A variable~$x$ in~$\bfw$ is \textit{simple} if $|\bfw|_x= 1$; otherwise, it is \textit{non-simple}.

An \textit{identity} is an expression $\bfu \approx \bfv$, where $\bfu, \bfv \in \scrA^*$; it is \textit{nontrivial} if $\bfu \neq \bfv$.
A monoid~$\monM$ \textit{satisfies} an identity $\bfu \approx \bfv$, indicated by $\monM \models \bfu \approx \bfv$, if for any substitution $\varphi:\scrA \to \monM$, the equality $\varphi\bfu = \varphi\bfv$ holds in~$\monM$.
A class~$\bbM$ of monoids \textit{satisfies} an identity $\bfu \approx \bfv$, indicated by $\bbM \models \bfu \approx \bfv$, if $\monM \models \bfu \approx \bfv$ for all $\monM \in \bbM$.

An identity $\bfu \approx \bfv$ is \textit{directly deducible} from an identity $\bfp \approx \bfq$ if some substitution $\varphi: \scrA \to \scrA^*$ and words $\bfa,\bfb \in \scrA^*$ exist such that $\{ \bfu,\bfv \} = \{ \bfa(\varphi\bfp)\bfb,\, \bfa(\varphi\bfq) \bfb \}$.
An identity $\bfu \approx \bfv$ is \textit{deducible} from a set~$\Sigma$ of identities, indicated by $\Sigma \vdash \bfu \approx \bfv$, if there exists a sequence $\bfu = \bfw_0,\, \bfw_1, \, \ldots, \, \bfw_k = \bfv$ of distinct words such that each identity $\bfw_i \approx \bfw_{i+1}$ is directly deducible from some identity in~$\Sigma$.
Informally, a deduction $\Sigma \vdash \bfu \approx \bfv$ holds if the identities in~$\Sigma$ can be used to convert~$\bfu$ into~$\bfv$.
Two sets of identities~$\Sigma_1$ and~$\Sigma_2$ are \textit{equivalent}, indicated by $\Sigma_1 \sim \Sigma_2$, if the deductions $\Sigma_1 \vdash \Sigma_2$ and $\Sigma_2 \vdash \Sigma_1$ hold.

\subsection{Some basic properties of varieties}

For any set~$\Sigma$ of identities, the variety \textit{defined by~$\Sigma$}, denoted by $\var\Sigma$, is the class of all monoids that satisfy every identity in~$\Sigma$.
A variety is \textit{finitely based} if it is defined by some finite set of identities; otherwise, it is \textit{non-finitely based}.
The subvariety of a variety~$\bfV$ defined by~$\Sigma$ is $\bfV\Sigma = \bfV \cap \var\Sigma$.

For any monoid~$\monM$, the variety \textit{generated by}~$\monM$ is the smallest variety containing~$\monM$.
A variety is \textit{finitely generated} if it is generated by a finite monoid; it is \textit{locally finite} if every finitely generated member is finite.
It is well known that every finitely generated variety is locally finite; see, for instance, Burris and Sankappanavar \cite[Theorem~II.10.16]{BurSan81}.

\begin{lemma} \label{L: LF FG max}
\begin{enumerate}[\rm(i)]
\item Any variety that satisfies any of the following identities is locally finite\textup: \begin{align*} x^2hx & \approx xhx, \tag{$\xxyx$} \label{id: xxyx=xyx} \\ xhx^2 & \approx xhx. \tag{$\xyxx$} \label{id: xyxx=xyx} \end{align*}
\item Every locally finite variety that is small is also finitely generated.
\item Every finitely generated variety has finitely positively many maximal subvarieties.
\end{enumerate}
\end{lemma}

\begin{proof}
(i) This follows from a general result of M.V.\ Sapir \cite[Proposition~3.1]{SapMV87}.

(ii) See Jackson and Lee \cite[Lemma~2.1]{JacLee18}.

(iii) See Lee \etal. \cite[Proposition~4.1]{LeeRhoSte19}.
\end{proof}

The \textit{dual} of a monoid~$\monM$, denoted by~$\overleftarrow{\monM}$, is the monoid obtained by reversing the multiplication of~$\monM$; in other words, the multiplication tables of~$\monM$ and~$\overleftarrow{\monM}$ are transposes of one another.
The \textit{dual} of a variety~$\bfV$ is $\overleftarrow{\bfV} = \{ \overleftarrow{\monM} \,|\, \monM \in \bfV \}$.
Equivalently, if~$\bfV$ is defined by some set $\{ \bfu_i \approx \bfv_i \,|\, i \in I \}$ of identities, then its dual~$\overleftarrow{\bfV}$ is defined by $\{ \overleftarrow{\,\bfu_i} \approx \overleftarrow{\,\bfv_i} \,|\, i \in I \}$, where~$\overleftarrow{\,\bfu_i}$ and~$\overleftarrow{\,\bfv_i}$ are, respectively, the words~$\bfu_i$ and~$\bfv_i$ written in reverse order.
A variety~$\bfV$ is \textit{self-dual} if $\bfV = \overleftarrow{\bfV}$.

A variety~$\bfV$ is \textit{small} if its lattice $\frakL(\bfV)$ of subvarieties is finite.
The \textit{interva}l $[\bfU,\bfV]$ is the lattice of all subvarieties of~$\bfV$ containing~$\bfU$; in particular, $\frakL(\bfV) = [\bftrivial,\bfV]$ where~$\bftrivial$ is the variety of trivial monoids.

\subsection{Varieties of \texorpdfstring{{\Jtrivial}}{J-trivial} monoids} 

Recall that~$\bbJ$ denotes the class of all {\Jtrivial} monoids.
The definition of subvarieties of~$\bbJ$ involves the \textit{aperiodicity identities} \begin{equation*} x^{n+1} \approx x^n, \quad n \geq 0 \tag*{$\aper$} \label{id: aperiodic} \end{equation*} and the \textit{eventual commutativity identities} \begin{equation*} (xy)^n \approx (yx)^n, \quad n \geq 0. \tag*{$\ecom$} \label{id: eventually} \end{equation*} 
Specifically, let~$\text{\ref{id: aperiodic}}_n$ and~$\text{\ref{id: eventually}}_n$ denote the $n$th identities in~\ref{id: aperiodic} and~\ref{id: eventually}, respectively, and let \[ \bfJ_n = \var\{\text{\ref{id: aperiodic}}_n,\, \text{\ref{id: eventually}}_n\}. \]
Note that $\bfJ_0 = \var\{x \approx 1\} = \bftrivial$ and $\bfJ_1 = \var\{x^2 \approx x,\; xy \approx yx\}$ is the variety of \textit{semilattice} monoids, that is, monoids that are idempotent and commutative.

\begin{lemma} \label{L: subvarieties of J}
\begin{enumerate}[\rm(i)] 
\item The inclusions $\bfJ_0 \subset \bfJ_1 \subset \bfJ_2 \subset \cdots \subset \bbJ$ hold and are proper.
\item The class~$\bbJ$ is not a variety\textup, but every subvariety of~$\bbJ$ is a subvariety of~$\bfJ_n$ for all sufficiently large $n \geq 0$.
\end{enumerate}
\end{lemma}

\begin{proof}
(i) This is routinely verified.

(ii) See Gusev and Sapir \cite[Fact~2.1]{GusSap22}.
\end{proof}

A word~$\mathbf w$ is an \textit{isoterm} for a variety~$\bfV$ if~$\bfV$ does not satisfy any nontrivial identity of the form $\bfw \approx \bfw'$. 
A monoid that is a union of groups is \textit{completely regular}.
A variety is \textit{completely regular} if it consists of completely regular monoids.
It is well known that a periodic variety is completely regular if and only if it satisfies the identity $x^{n+1} \approx x$ for some $n \geq 1$.

\begin{lemma} \label{L: comm CR}
Let~$\bfV$ be any subvariety of~$\bbJ$. 
Then
\begin{enumerate}[\rm(i)]
\item $\bfV$ is completely regular if and only if $\bfV \subseteq \bfJ_1$\textup;
\item $\bfV$ is commutative if and only if~$xy$ is not an isoterm for~$\bfV$\textup;
\item $\bfV$ is commutative implies that~$\bfV$ is Cross and $\bfV = \var\{ \text{\ref{id: aperiodic}}_n,\, xy \approx yx \}$ for some $n \geq 0$.
\end{enumerate}
\end{lemma}

\begin{proof}
By Lemma~\ref{L: subvarieties of J}(ii), the variety~$\bfV$ satisfies the identities $\{ \text{\ref{id: aperiodic}}_k, \text{\ref{id: eventually}}_k \}$ for some $k \geq 2$.

(i) Suppose that~$\bfV$ is completely regular and so satisfies the identity $\sigma: x^{n+1} \approx x$ for some $n \geq 1$.
Then~$\bfV$ satisfies~\ref{id: aperiodic}$_1$ because \[ \bfV \models x \stackrel{\sigma}{\approx} x^{kn+1} \stackrel{\text{\ref{id: aperiodic}}_k}{\approx} x^{kn+2} \stackrel{\sigma}{\approx} x^2 \vdash \text{\ref{id: aperiodic}}_1, \] and~$\bfV$ satisfies \ref{id: eventually}$_1$ because \[ \bfV \models xy \stackrel{\text{\ref{id: aperiodic}}_1}{\approx} (xy)^k \stackrel{\text{\ref{id: eventually}}_k}{\approx} (yx)^k \stackrel{\text{\ref{id: aperiodic}}_1}{\approx} yx. \]
Therefore, the inclusion $\bfV \subseteq \bfJ_1$ holds.
Conversely, every idempotent monoid, and so every subvariety of~$\bfJ_1$, is completely regular.

(ii) If~$\bfV$ is commutative, then it satisfies $xy \approx yx$ and so~$xy$ is not an isoterm for~$\bfV$.
Conversely, suppose that~$xy$ is not an isoterm for~$\bfV$.
Then it is easily shown that~$\bfV$ is either commutative or idempotent; see, for instance, Gusev and Sapir \cite[Lemma~2.7]{GusSap22}.
If~$\bfV$ is idempotent, then it is completely regular and so is commutative by part~(i).

(iii) This follows from the complete description of all commutative varieties~\cite{Hea68}.
\end{proof}

\subsection{Rees quotients of free monoids}

For any set $\scrW \subseteq \scrA^*$, let $\monRq\,\scrW$ denote the Rees quotient of~$\scrA^*$ over the ideal of all words that are not factors of any word in~$\scrW$.
Equivalently, $\monRq\,\scrW$ is the monoid that consists of every factor of every word in~$\scrW$, together with a zero element~$0$, with binary operation~$\cdot$ given by \[ \bfu \cdot \bfv = \begin{cases} \bfu\bfv & \mbox{if $\bfu\bfv$ is a factor of some word in~$\scrW$}, \\ 0 & \mbox{otherwise}. \end{cases} \]
If there is a uniform upper bound~$n$ on the length of the words in~$\scrW$, then $\monRq\,\scrW$ satisfies the identities $\{ \text{\ref{id: aperiodic}}_{n+1}, \text{\ref{id: eventually}}_{n+1} \}$ and so is {\Jtrivial}; in particular, $\monRq\,\scrW$ is {\Jtrivial} for all finite $\scrW \subseteq \scrA^*$.

Let $\bfRq\,\scrW$ denote the variety generated by $\monRq\,\scrW$. 
Some easy examples of Rees quotients of~$\scrA^*$ that will appear later in this article are 
\begin{align*}
\monRq\{x^n\} & = \langle x,1 \,|\, x^{n+1} = 0 \rangle = \{ 0,\, x,\, x^2,\, \ldots,\, x^n,\, 1\}, \\ 
\monRq\{xy\} & = \langle x,y,1 \,|\, x^2=y^2=yx=0 \rangle = \{ 0,\, x,\, y,\, xy,\, 1\}, \\
\monRq\{ xhx\} & = \langle x,h,1 \,|\, x^2=h^2=hxh=0\rangle = \{ 0,\, x,\, h,\, xh,\, hx,\, xhx,\, 1\}.
\end{align*}
Since the monoid~$\monRq\{x^n\}$ is commutative, it is easily shown that \[ \bfRq\{x^n\} = \var\{ x^{n+2} \approx x^{n+1}, \; xy \approx yx \}; \] see Almeida \cite[Corollary~6.1.5]{Alm94}.
In particular, $\bfRq\{x^0\} = \bfRq\{1\} = \bfJ_1$ coincides with the variety of semilattice monoids.
The other two varieties can be found in Jackson~\cite{Jac05}: \begin{align*} \bfRq\{xy\} & = \var\{ xyx \approx x^2y, \; xyx \approx yx^2 \}, \\ \bfRq\{xhx\} & = \var\bigg\{ \begin{array}{l} x^2h \approx hx^2, \; xhxtx \approx x^2ht, \; xyhxty \approx yxhxty, \\ xhxyty \approx xhyxty, \; xhytxy \approx xhytyx \end{array} \bigg\}. \end{align*}

\begin{lemma}[{Jackson \cite[Lemma~3.3]{Jac05}}] \label{L: isoterm}
For any variety~$\bfV$ and any set~$\scrW$ of words\textup, the inclusion $\bfRq\,\scrW \subseteq \bfV$ holds if and only if every word in~$\scrW$ is an isoterm for~$\bfV$.
\end{lemma}

\begin{lemma} \label{L: xyx not isoterm}
\begin{enumerate}[\rm(i)]
\item Suppose that~$\bfV$ is any subvariety of~$\bbJ$ such that $\monRq\{xhx\} \notin \bfV$.
Then~$\bfV$ is a subvariety of one of the following varieties\textup: 
\begin{equation}
\bfJ_2\{\eqref{id: xxyx=xyx}\}, \quad \bfJ_2\{\eqref{id: xyxx=xyx}\}, \quad \bfJ_n\{xhx \approx x^2h\}, \quad \bfJ_n\{xhx \approx hx^2\}, \quad n \geq 2. \label{D: xyx not isoterm}
\end{equation}
\item For each $n \geq 1$\textup, the varieties $\var\{ \text{\ref{id: aperiodic}}_n,\, xhx \approx x^2h \}$ and $\var\{ \text{\ref{id: aperiodic}}_n,\, xhx \approx hx^2 \}$ are Cross.
\end{enumerate}
\end{lemma}

\begin{proof}
(i) By Lemma~\ref{L: subvarieties of J}(ii), the variety~$\bfV$ satisfies the identities $\{ \text{\ref{id: aperiodic}}_n, \text{\ref{id: eventually}}_n \}$ for some $n \geq 2$.
By Lemma~\ref{L: isoterm}, the assumption $\monRq\{xhx\} \notin \bfV$ implies that~$\bfV$ satisfies a nontrivial identity $\sigma: xhx \approx \bfw$ for some $\bfw \in \scrA^*$.
Let $p = |\bfw|_x$ and $q = |\bfw|_h$.
There are two cases.

\medskip

\noindent\textsc{Case~1}: $p \leq 1$ or $q \neq 1$ or there exists $z \in \scrA \backslash\{x,h\}$ such that $|\bfw|_z = r \geq 1$. 
Then~$\bfV$ satisfies $x^2 \approx x^p$ or $h \approx h^q$ or $1 \approx z^r$, whence~$\bfV$ is completely regular.
Therefore, the inclusion $\bfV \subseteq \bfJ_1$ holds by Lemma~\ref{L: comm CR}(i), so that~$\bfV$ is a subvariety of all the varieties in~\eqref{D: xyx not isoterm}.

\medskip

\noindent\textsc{Case~2}: $p \geq 2$ and $q = 1$ with $\content(\bfw) = \{ x,y\}$. 
Then $\bfw = x^{p_1}hx^{p_2}$ for some $p_1,p_2 \geq 0$ such that $p_1+p_2 = p \geq 2$ and $(p_1,p_2) \neq (1,1)$.
If $p = 2$, then $(p_1,p_2) \in \{ (2,0),(0,2)\}$, so that either $\bfV \subseteq \bfJ_n\{xhx \approx x^2h\}$ or $\bfV \subseteq \bfJ_n\{xhx \approx hx^2\}$.
Therefore, suppose that $p_1+p_2 = p \geq 3$, so that either $p_1 \geq 2$ or $p_2 \geq 2$.
Then~$\bfV$ satisfies the identity $\tau: x^2 \approx x^{2+k}$ with $k = p-2$ and so also the identities $\{ \text{\ref{id: aperiodic}}_2, \text{\ref{id: eventually}}_2 \}$ because
\begin{align*}
\bfV & \models x^2 \stackrel{\tau}{\approx} x^{2+nk} \stackrel{\text{\ref{id: aperiodic}}_n}{\approx} x^{3+nk} \stackrel{\tau}{\approx} x^3 \vdash \text{\ref{id: aperiodic}}_2 \\
\text{and } \; \bfV & \models (xy)^2 \stackrel{\text{\ref{id: aperiodic}}_2}{\approx} (xy)^n \stackrel{\text{\ref{id: eventually}}_n}{\approx} (yx)^n \stackrel{\text{\ref{id: aperiodic}}_2}{\approx} (yx)^2 \vdash \text{\ref{id: eventually}}_2.
\end{align*}
Now if $p_1 \geq 2$, then \[ \bfV \models xhx \stackrel{\sigma}{\approx} x^{p_1}hx^{p_2} \stackrel{\text{\ref{id: aperiodic}}_2}{\approx} x x^{p_1}hx^{p_2} \stackrel{\sigma}{\approx} xxhx \vdash \eqref{id: xxyx=xyx}; \]
similarly, if $p_2 \geq 2$, then $\bfV \models \eqref{id: xyxx=xyx}$.
Consequently, either $\bfV \subseteq \bfJ_2\{\eqref{id: xxyx=xyx}\}$ or $\bfV \subseteq \bfJ_2\{\eqref{id: xyxx=xyx}\}$.

\medskip

(ii) This is Lee \cite[Corollary~3.6]{Lee14rm}.
\end{proof}

\subsection{The variety \texorpdfstring{$\bfO \cap \bfJ_2$}{O J2} } \label{subsec: O}

Let~$\bfO$ denote the variety defined by the identities \begin{equation} xhxyxty \approx xhyxty, \quad xhytxyx \approx xhytyx. \label{id: O} \end{equation}
Since the identity~\eqref{id: xyxx=xyx} is directly deducible from the first identity in~\eqref{id: O}, the variety~$\bfO$ is locally finite by Lemma~\ref{L: LF FG max}(i).

The variety $\bfO \cap \bfJ_2$ plays an important role later in this article in the description of some Cross varieties.
The main goal of this subsection is to show that every noncommutative subvariety of $\bfO \cap \bfJ_2$ can be defined by identities of very restricted types (Proposition~\ref{P: subvarieties of O}).
Some Cross subvarieties of $\bfO \cap \bfJ_2$ will also be identified.

\begin{lemma} \label{L: O sub}
\begin{enumerate}[\rm(i)]
\item Every subvariety of~$\bfO$ for which~$xy$ is an isoterm is defined by~\eqref{id: O} and finitely many identities from the following sets\textup:
\begin{align} 
& \Bigg\{ \; x^{e_0} \prod_{i=1}^m (h_ix^{e_i}) \approx x^{f_0} \prod_{i=1}^m (h_ix^{f_i}) \; \; \Bigg| \begin{array}{l} e_0,f_0, e_1,f_1,\ldots,e_m,f_m \geq 0; \\[0.05in] \sum_{i=0}^me_i,\, \sum_{i=0}^mf_i \geq 2;\; m \geq 0 \end{array} \Bigg\}; \label{id: O sub xhxhx} \\
& \left\{ \; \bfh\bfp\bfc \approx \bfh\bfq\bfc \; \; \left| 
\begin{array}{l} 
\bfh \in \{ 1,\, yh\}; \; \bfp,\bfq \in \{x,y\}^+; \\[0.05in] 
\bfc \in \big\{ 1,\, txy,\, \prod_{i=1}^m (t_i\bfc_i) \,\big|\, \bfc_i \in \{ 1,x,y\},\, m \geq 1 \big\}; \\[0.05in] 
1 \leq |\bfp|_x = |\bfq|_x,\, |\bfp|_y = |\bfq|_y \leq 2; \\[0.05in] 
|\bfh\bfp\bfc|_x = |\bfh\bfq\bfc|_x,\, |\bfh\bfp\bfc|_y = |\bfh\bfq\bfc|_y \geq 2 
\end{array} 
\right. \right\}. \label{id: O sub yhpc=yhqc}
\end{align}
\item Every noncommutative subvariety of $\bfO \cap \bfJ_2$ is defined by $\{ \eqref{id: O}, \text{\ref{id: aperiodic}}_2, \text{\ref{id: eventually}}_2 \} \cup \Sigma_1 \cup \Sigma_2$ for some $\Sigma_1 \subseteq \eqref{id: O sub xhxhx}$ and $\Sigma_2 \subseteq \eqref{id: O sub yhpc=yhqc}$.
\end{enumerate}
\end{lemma}

\begin{proof}
(i) This can be extracted from Gusev and Sapir \cite[proof of Proposition~3.1]{GusSap22}. 

(ii) This follows from Lemma~\ref{L: comm CR}(ii) and part~(i).
\end{proof}

\begin{lemma} \label{L: Oc sub}
For any $\Sigma_1 \subseteq \eqref{id: O sub xhxhx}$\textup, the equivalence $\eqref{id: xyxx=xyx} \cup \Sigma_1 \sim \eqref{id: xyxx=xyx} \cup \Sigma$ holds for some possibly empty set~$\Sigma$ that consists of some of the following identities\textup:
\begin{equation}
xhx \approx hx^2, \quad x^2hx \approx x^2h, \quad x^2hxtx \approx x^2htx, \quad \eqref{id: xxyx=xyx}. \label{id: Oc sub}
\end{equation}
\end{lemma}

\begin{proof}
It follows from Gusev \cite[Lemma~6.7]{Gus25sf} that the equivalence \[ \{ \text{\ref{id: aperiodic}}_2,\, x^2hx^2 \approx x^2hx \} \cup \Sigma_1 \sim \{ \text{\ref{id: aperiodic}}_2,\, x^2hx^2 \approx x^2hx \} \cup \Sigma \] holds for some possibly empty set~$\Sigma$ that consists of some of the following identities\textup:
\begin{enumerate}[\ (a)]
\item $xhx^2 \approx hx^2$,\, $x^2hx \approx x^2h$,\, $x^2hxtx \approx x^2htx$;
\item $\big(\prod_{i=1}^{n-1} (xh_i)\big)xh_nx^2 \approx \big(\prod_{i=1}^{n-1} (xh_i)\big)xh_nx$ for some $n \geq 1$;
\item $\big(\prod_{i=1}^{n-1} (xh_i)\big)x^2h_nx \approx \big(\prod_{i=1}^{n-1} (xh_i)\big)xh_nx^2$ for some $n \geq 1$.
\end{enumerate}
(Note that if $n=1$, then the identities in~(b) and~(c) are $xh_1x^2 \approx xh_1x$ and $x^2h_1x \approx xh_1x^2$, respectively.)
It is easily seen that the deduction $\eqref{id: xyxx=xyx} \vdash \{\text{\ref{id: aperiodic}}_2,\, x^2hx^2 \approx x^2hx \}$ holds, so the equivalence $\eqref{id: xyxx=xyx} \cup \Sigma_1 \sim \eqref{id: xyxx=xyx} \cup \Sigma$ follows.
Since it is routinely checked that the equivalence $\{ \eqref{id: xyxx=xyx}, \rm(a), \rm(b), \rm(c) \} \sim \{ \eqref{id: xyxx=xyx}, \eqref{id: Oc sub} \}$ holds, the identities in~$\Sigma$ can be chosen from~\eqref{id: Oc sub}.
\end{proof}

\begin{proposition} \label{P: subvarieties of O}
Every noncommutative subvariety of $\bfO \cap \bfJ_2$ is defined by \[ \{ \eqref{id: O}, \text{\ref{id: aperiodic}}_2, \text{\ref{id: eventually}}_2 \} \cup \Sigma_1 \cup \Sigma_2 \] for some $\Sigma_1 \subseteq \eqref{id: Oc sub}$ and some subset~$\Sigma_2$ of
\begin{equation} 
\left\{ \; \bfh\bfp\bfc \approx \bfh\bfq\bfc \; \; \left| 
\begin{array}{l} 
\bfh \in \{ 1,\, yh\}; \; \bfp,\bfq \in \{x,y\}^+; \\[0.05in] 
\bfc \in \big\{ 1,\, txy,\, \prod_{i=1}^m (t_i\bfa_i), \,\prod_{i=2}^{m+1} (t_i\bfa_i) \,\big|\, m \geq 1 \big\}; \\[0.05in] 
1 \leq |\bfp|_x = |\bfq|_x, \, |\bfp|_y = |\bfq|_y \leq 2; \\[0.05in] 
|\bfh\bfp\bfc|_x = |\bfh\bfq\bfc|_x,\, |\bfh\bfp\bfc|_y = |\bfh\bfq\bfc|_y \geq 2 
\end{array} 
\right. \right\}, \label{id: O J2 sub yhpc=yhqc}
\end{equation}
where $(\bfa_1,\bfa_2,\bfa_3,\bfa_4,\ldots)$ is the alternating sequence $(x,y,x,y,\ldots)$.
\end{proposition}

\begin{proof}
Let~$\bfV$ be any noncommutative subvariety of $\bfO \cap \bfJ_2$.
Then by Lemma~\ref{L: O sub}(ii), the variety~$\bfV$ is defined by $\{ \eqref{id: O}, \text{\ref{id: aperiodic}}_2, \text{\ref{id: eventually}}_2 \} \cup \Sigma_1 \cup \Sigma_2$ for some $\Sigma_1 \subseteq \eqref{id: O sub xhxhx}$ and some $\Sigma_2 \subseteq \eqref{id: O sub yhpc=yhqc}$.
Since the deduction $\eqref{id: O} \vdash \eqref{id: xyxx=xyx}$ holds, it follows from Lemma~\ref{L: Oc sub} that the identities in~$\Sigma_1$ can be chosen from~\eqref{id: Oc sub}.
Hence it remains to show that the identities in~$\Sigma_2$ can be chosen from~\eqref{id: O J2 sub yhpc=yhqc}.

Let $\bfh\bfp\bfc \approx \bfh\bfq\bfc$ be any identity in $\Sigma_2 \subseteq \eqref{id: O sub yhpc=yhqc}$.
If $\bfc \in \{ 1,\, txy\}$, then by definition, the identity $\bfh\bfp\bfc \approx \bfh\bfq\bfc$ is already in~\eqref{id: O J2 sub yhpc=yhqc}.
Therefore, it suffices to consider the case when \[ \bfc = \prod_{i=1}^m (t_i\bfc_i) = t_1\bfc_1 t_2 \bfc_2 \cdots t_m \bfc_m, \] where $\bfc_1,\bfc_2,\ldots,\bfc_m \in \{ 1,x,y \}$ and $m \geq 1$. 
In the remainder of this proof, it is shown that if either $\bfc_k = 1$ or $\bfc_k = \bfc_{k+1} \in \{ x,y \}$ for some~$k$, then the equivalence \begin{equation} \{ \eqref{id: O},\, \bfh\bfp\bfc \approx \bfh\bfq\bfc\} \sim \{ \eqref{id: O},\, \bfh\bfp\bfc' \approx \bfh\bfq\bfc'\} \label{D: hpc <-> hpc'} \end{equation} holds, where~$\bfc'$ is obtained from~$\bfc$ by removing the factor~$t_k\bfc_k$.
It follows that the identity $\bfh\bfp\bfc \approx \bfh\bfq\bfc$ can be chosen so that $(\bfc_1,\bfc_2,\ldots,\bfc_m) \in \{ (x,y,x,y,\ldots),\, (y,x,y,x,\ldots) \}$; in other words, $\bfh\bfp\bfc \approx \bfh\bfq\bfc$ can be chosen from~\eqref{id: O J2 sub yhpc=yhqc}.

\medskip

\noindent\textsc{Case~1}: $\bfc_k = 1$ with $1 \leq k \leq m$.
If $\bfc_m=1$, so that $\bfc = \big( \prod_{i=1}^{m-1} (t_i\bfc_i) \big) t_m$ and $\bfc' = \prod_{i=1}^{m-1} (t_i\bfc_i)$, then it is obvious that $\bfh\bfp\bfc \approx \bfh\bfq\bfc \sim \bfh\bfp\bfc' \approx \bfh\bfq\bfc'$.
Thus assume that $1 \leq k < m$, so that \[ \bfc = \bigg( \prod_{i=1}^{k-1} (t_i\bfc_i) \bigg) t_k \bigg( \prod_{i=k+1}^m (t_i\bfc_i) \bigg) \ \text{ and } \ \bfc' = \bigg( \prod_{i=1}^{k-1} (t_i\bfc_i) \bigg) \bigg( \prod_{i=k+1}^m (t_i\bfc_i) \bigg). \]
The deduction $\bfh\bfp\bfc \approx \bfh\bfq\bfc \vdash \bfh\bfp\bfc' \approx \bfh\bfq\bfc'$ is obvious.
Conversely, since~$\bfc$ is obtained by making the substitution $t_{k+1} \mapsto t_kt_{k+1}$ in~$\bfc'$, the deduction $\bfh\bfp\bfc' \approx \bfh\bfq\bfc' \vdash \bfh\bfp\bfc \approx \bfh\bfq\bfc$ holds.
It follows that the equivalence~\eqref{D: hpc <-> hpc'} holds.

\medskip

\noindent\textsc{Case~2}: $\bfc_k = \bfc_{k+1} \in \{ x,y\}$ with $1 \leq k < k+1 \leq m$.
Then \[ \bfc = \bigg( \prod_{i=1}^{k-1} (t_i\bfc_i) \bigg) t_k\bfc_k t_{k+1} \bfc_k \bigg( \prod_{i=k+2}^m (t_i\bfc_i) \bigg) \ \text{ and } \ \bfc' = \bigg( \prod_{i=1}^{k-1} (t_i\bfc_i) \bigg) t_{k+1}\bfc_k \bigg( \prod_{i=k+2}^m (t_i\bfc_i) \bigg). \]
Deleting the variable~$t_{k+1}$ from both sides of $\bfh\bfp\bfc \approx \bfh\bfq\bfc$ followed by making the substitution $t_k \mapsto t_{k+1}$ results in the identity \[ \bfh\bfp\bigg( \prod_{i=1}^{k-1} (t_i\bfc_i) \bigg) t_{k+1}\bfc_k^2 \bigg( \prod_{i=k+2}^m (t_i\bfc_i) \bigg) \approx \bfh\bfq\bigg( \prod_{i=1}^{k-1} (t_i\bfc_i) \bigg) t_{k+1}\bfc_k^2 \bigg( \prod_{i=k+2}^m (t_i\bfc_i) \bigg); \]
since $\bfc_k \in \{ x,y \} = \content(\bfp) = \content(\bfq)$, the identity~\eqref{id: xyxx=xyx} deducible from~\eqref{id: O} can be used to eliminate from both sides the occurrence of~$\bfc_k$ that immediately follows~$t_{k+1}$, resulting in $\bfh\bfp\bfc' \approx \bfh\bfq\bfc'$.
The deduction $\{ \eqref{id: O},\, \bfh\bfp\bfc \approx \bfh\bfq\bfc \} \vdash \bfh\bfp\bfc' \approx \bfh\bfq\bfc'$ thus holds.
Conversely, making the substitution $t_{k+1} \mapsto t_k\bfc_kt_{k+1}$ on both sides of $\bfh\bfp\bfc' \approx \bfh\bfq\bfc'$ results in $\bfh\bfp\bfc \approx \bfh\bfq\bfc$, so the deduction $\bfh\bfp\bfc' \approx \bfh\bfq\bfc' \vdash \bfh\bfp\bfc \approx \bfh\bfq\bfc$ holds.
Hence the equivalence~\eqref{D: hpc <-> hpc'} holds.
\end{proof}

Although not every subvariety of $\bfO \cap \bfJ_2$ is Cross, some identities from \eqref{id: Oc sub} and \eqref{id: O J2 sub yhpc=yhqc} do define Cross subvarieties of $\bfO \cap \bfJ_2$.
For instance, the identity $x^2hxtx \approx x^2htx$ from \eqref{id: Oc sub} and the \textit{restrictive identities} from \eqref{id: O J2 sub yhpc=yhqc}: \begin{equation} xy \prod_{i=1}^m (t_i\bfa_i) \approx yx \prod_{i=1}^m (t_i\bfa_i), \quad m \geq 2. \tag*{$\Osub$} \label{id: Hsub} \end{equation} 
Let~$\text{\ref{id: Hsub}}_m$ denote the $m$th identity in~\ref{id: Hsub}.

\begin{corollary} \label{C: O cap J2}
\begin{enumerate}[\rm(i)]
\item The variety $\bfO \cap \bfJ_2\{ x^2hxtx \approx x^2htx \}$ is Cross.
\item For each $m \geq 2$\textup, the variety $\bfO \cap \bfJ_2\{ \text{\ref{id: Hsub}}_m \}$ is Cross.
\end{enumerate}
\end{corollary}

\begin{proof}
Let $\sigma \in \{ x^2hxtx \approx x^2htx,\, \text{\ref{id: Hsub}}_m \}$.
Since $\bfO \models \eqref{id: O} \vdash \eqref{id: xyxx=xyx}$, the finitely based variety $\bfO \cap \bfJ_2\{ \sigma \}$ is locally finite by Lemma~\ref{L: LF FG max}(i).
In what follows, this variety is shown to be small and so is Cross by Lemma~\ref{L: LF FG max}(ii).

Let~$\bfV$ be any subvariety of $\bfO \cap \bfJ_2\{ \sigma \}$.
By Lemma~\ref{L: comm CR}(iii), only three subvarieties of $\bfO \cap \bfJ_2\{ \sigma \}$ are commutative; specifically, they are $[ \text{\ref{id: aperiodic}}_n,\, xy \approx yx]$, where $n \in \{ 0,1,2\}$.
Therefore, assume that~$\bfV$ is not commmutative, so that by Proposition~\ref{P: subvarieties of O}, it is defined by $\{ \eqref{id: O}, \text{\ref{id: aperiodic}}_2, \text{\ref{id: eventually}}_2,\, \sigma \} \cup \Sigma_1 \cup \Sigma_2$ for some $\Sigma_1 \subseteq \eqref{id: Oc sub}$ and $\Sigma_2 \subseteq \eqref{id: O J2 sub yhpc=yhqc}$.
Since \eqref{id: Oc sub} is finite, there are only finitely many choices for the set~$\Sigma_1$.
As for the set~$\Sigma_2$, there are infinitely many choices due precisely to identities from~\eqref{id: O J2 sub yhpc=yhqc} of the following types: \[ \lambda_m: \bfh\bfp \prod_{i=1}^m (t_i\bfa_i) \approx \bfh\bfq \prod_{i=1}^m (t_i\bfa_i) \ \text{ and } \ \rho_m: \bfh\bfp \prod_{i=2}^{m+1} (t_i\bfa_i) \approx \bfh\bfq \prod_{i=2}^{m+1} (t_i\bfa_i), \quad m \geq 1. \]
However, in each of the following cases, it is shown that whenever~$\lambda_m$ or~$\rho_m$ is required for the definition of~$\bfV$, it is possible to choose~$m$ from a fixed finite set.
It follows that there are only finitely many choices for~$\Sigma_2$.
Consequently, there are only finitely many choices for~$\bfV$, whence $\bfO \cap \bfJ_2\{ \sigma \}$ is small.

\medskip

\noindent\textsc{Case~1}: $\sigma$ is $x^2hxtx \approx x^2htx$.
Consider the identity~$\lambda_m$ with $m \geq 7$.
Removing the variables~$t_3$ and~$t_4$ from both sides of~$\lambda_m$ results in 
\begin{equation}
\bfh\bfp \cdot t_1x \cdot t_2y \cdot x \cdot y \cdot t_5x \cdot t_6y \prod_{i=7}^m (t_i\bfa_i) \approx \bfh\bfq \cdot t_1x \cdot t_2y \cdot x \cdot y \cdot t_5x \cdot t_6y \prod_{i=7}^m (t_i\bfa_i); \label{id: hpc=hqc 1}
\end{equation}
in other words, $\lambda_m \vdash \eqref{id: hpc=hqc 1}$.
Recall that $\eqref{id: O} \vdash \eqref{id: xyxx=xyx}$ and $x,y \in \content(\bfp) = \content(\bfq)$; hence the deduction $\{\eqref{id: O},\sigma,\lambda_m\} \vdash \lambda_{m-2}$ holds because
\begin{align*}
\{\eqref{id: O},\sigma,\lambda_m\} \vdash
\bfh\bfp \cdot t_1x \cdot t_2y \cdot t_5x \cdot t_6y \prod_{i=7}^m (t_i\bfa_i) & \stackrel{\eqref{id: xyxx=xyx}}{\approx} \bfh\bfp \cdot t_1x^2 \cdot t_2y^2 \cdot t_5x \cdot t_6y \prod_{i=7}^m (t_i\bfa_i) \\
& \stackrel{\sigma}{\approx} \bfh\bfp \cdot t_1x^2 \cdot t_2y^2 \cdot x \cdot y \cdot t_5x \cdot t_6y \prod_{i=7}^m (t_i\bfa_i) \\
& \stackrel{\eqref{id: hpc=hqc 1}}{\approx} \bfh\bfq \cdot t_1x^2 \cdot t_2y^2 \cdot x \cdot y \cdot t_5x \cdot t_6y \prod_{i=7}^m (t_i\bfa_i) \\
& \stackrel{\sigma}{\approx} \bfh\bfq \cdot t_1x^2 \cdot t_2y^2 \cdot t_5x \cdot t_6y \prod_{i=7}^m (t_i\bfa_i) \\
& \stackrel{\eqref{id: xyxx=xyx}}{\approx} \bfh\bfq \cdot t_1x \cdot t_2y \cdot t_5x \cdot t_6y \prod_{i=7}^m (t_i\bfa_i) \vdash \lambda_{m-2}.
\end{align*}
The deduction $\lambda_{m-2} \vdash \lambda_m$ is obvious, so the equivalence $\{\eqref{id: O},\sigma,\lambda_m\} \sim \{\eqref{id: O},\sigma,\lambda_{m-2}\}$ follows.
If $m - 2 \geq 7$, then the same argument gives $\{\eqref{id: O},\sigma,\lambda_{m-2}\} \sim \{\eqref{id: O},\sigma,\lambda_{m-4}\}$.
This can be repeated to obtain \[ \{\eqref{id: O},\sigma,\lambda_m\} \sim \{\eqref{id: O},\sigma,\lambda_{m-2}\} \sim \cdots \sim \{\eqref{id: O},\sigma,\lambda_k\} \] for some $k \leq 6$.
Similarly, $\{\eqref{id: O},\sigma,\rho_m\} \sim \{\eqref{id: O},\sigma,\rho_k\}$ for some $k \leq 6$.
Therefore, whenever any of the identities~$\lambda_m$ and~$\rho_m$ is required for the definition of~$\bfV$, it is possible to choose it from $\{\lambda_k,\rho_k \,|\, 1 \leq k \leq 6\}$.

\medskip

\noindent\textsc{Case~2}: $\sigma$ is $\text{\ref{id: Hsub}}_m$.
Since $\bfp,\bfq \in \{x,y\}^+$ are such that $1 \leq |\bfp|_x = |\bfq|_x, \, |\bfp|_y = |\bfq|_y \leq 2$, it is easily seen that the deduction $\text{\ref{id: Hsub}}_m \vdash \{ \lambda_k, \rho_k\}$ holds for all $k \geq m$.
Therefore, whenever any identity from $\{\lambda_k,\rho_k \,|\, k \geq 1\}$ is required for the definition of~$\bfV$, it is possible to choose it from $\{\lambda_k,\rho_k \,|\, 1 \leq k \leq m-1\}$.
\end{proof}

\subsection{The variety \texorpdfstring{$\bfAzero \vee \bfQ$}{A0 v Q}} \label{subsec: A v Q}

Let $\bfAzero$, $\bfBzero$, $\bfE$, and $\bfQ$ denote the varieties generated by the monoids
\begin{align*}
\monAzero & = \langle e,f,1 \,|\, e^2=e,\, f^2=f,\, fe=0 \rangle = \{ 0,\, e,\, f,\, ef,\, 1\}, \\
\monBzero & = \langle a,e,f,1 \,|\, af=ea=a,\, e^2=e,\, f^2=f,\, ef=fe=0 \rangle = \{ 0,\, a,\, e,\, f,\, 1\}, \\
\monE & = \langle a,e,1 \,|\, ae=a,\, ea=0,\, e^2=e \rangle = \{ 0,\, a,\, e,\, 1\}, \\ 
\text{and }\, \monQ & = \langle a, b, e, 1 \,|\, ae=ba=eb=0,\, be=b,\, ea=a,\, e^2=e \rangle = \{ 0,\, a,\, b,\, e,\, ab,\, 1 \},
\end{align*}
respectively.
These monoids satisfy the identities $\{ \text{\ref{id: aperiodic}}_2, \text{\ref{id: eventually}}_2 \}$ and so $\bfAzero,\bfBzero,\bfE,\bfQ \subseteq \bfJ_2$.

\begin{lemma}[Lee \cite{Lee08,Lee14rm}] \label{L: bases A B E Q}
The varieties~$\bfAzero$\textup, $\bfBzero$\textup, $\bfE$\textup, and~$\bfQ$ are Cross.
Specifically\textup,
\begin{enumerate}[\rm(i)] 
\item $\bfAzero = \var\{ xhxtx \approx xhtx,\, \text{\ref{id: eventually}}_2 \} = \var\{ xhxtx \approx xhtx, \; xyhxty \approx yxhxty, \; xhytxy \approx xhytyx \}$\textup;
\item $\bfBzero = \var\{ xhxtx \approx xhtx,\; x^2y^2 \approx y^2x^2 \}$\textup;
\item $\bfE = \var\{\text{\ref{id: aperiodic}}_2,\; xhx \approx x^2h,\; x^2y^2 \approx y^2x^2\} = \bfJ_2 \{ xhx \approx x^2h \}$\textup;
\item $\bfQ = \var\{\eqref{id: xxyx=xyx}, \eqref{id: xyxx=xyx},\; x^2y^2 \approx y^2x^2\}$. 
\end{enumerate}
\end{lemma}

It is easily seen that the varieties~$\bfAzero$, $\bfBzero$, and~$\bfQ$ are self-dual but the variety~$\bfE$ is not.

\begin{lemma} \label{L: exclusion A0 Q B0}
Let~$\bfV$ be any locally finite subvariety of~$\bfJ_n$ for some $n \geq 2$.
Then
\begin{enumerate}[\rm(i)]
\item $\monAzero \notin \bfV$ if and only if $\bfV \models (x^ny^n)^2 \approx x^ny^n$\textup;
\item $\monQ \notin \bfV$ if and only if $\bfV \models x^nhxtx^n \approx x^nhtx^n$\textup;
\item $\monBzero \notin \bfV$ if and only if either $\bfV \models x^2 \approx x$\textup, $\bfV \models x^2hx^2 \approx x^2h$\textup, or $\bfV \models x^2hx^2 \approx hx^2$.
\end{enumerate}
\end{lemma}

\begin{proof}
(i) This follows from Lee \etal. \cite[Theorem~5.23]{LeeRhoSte19}.

(ii) If~$\bfV$ is completely regular, then $\bfV \subseteq \bfJ_1$ by Lemma~\ref{L: comm CR}(i), so that both $\monQ \notin \bfV$ and $\bfV \models x^nhxtx^n \approx x^nhtx^n$ hold.
If~$\bfV$ is not completely regular, then the equivalence holds by Gusev \cite[Lemma~2.12]{Gus25ijac}.

(iii) This follows from Almeida \cite[Proposition~11.10.2]{Alm94}.
\end{proof}

\begin{corollary} \label{C: exclusion A0 Q}
Let~$\bfV$ be any subvariety of $\bfJ_2\{ \eqref{id: xxyx=xyx}, \eqref{id: xyxx=xyx} \}$.
Then
\begin{enumerate}[\rm(i)]
\item $\monAzero \notin \bfV$ if and only if $\bfV \subseteq \bfQ$\textup;
\item $\monQ \notin \bfV$ if and only if $\bfV \subseteq \bfAzero$\textup;
\item $\bfAzero \vee \bfQ \subseteq \bfV$ implies that $\frakL(\bfV) = \frakL(\bfAzero \vee \bfQ) \cup [\bfAzero \vee \bfQ,\bfV]$.
\end{enumerate}
\end{corollary}

\begin{proof}
Since $\bfV \models \{ \text{\ref{id: aperiodic}}_2, \text{\ref{id: eventually}}_2, \eqref{id: xxyx=xyx}, \eqref{id: xyxx=xyx} \}$ by assumption, $\bfV$ is locally finite by Lemma~\ref{L: LF FG max}(i).

(i) Suppose that $\monAzero \notin \bfV$.
Then~$\bfV$ satisfies $\sigma: (x^2y^2)^2 \approx x^2y^2$ by Lemma~\ref{L: exclusion A0 Q B0}(i), so that \[ \bfV \models x^2y^2 \stackrel{\sigma}{\approx} (x^2y^2)^2 \stackrel{\text{\ref{id: eventually}}_2}{\approx} (y^2x^2)^2 \stackrel{\sigma}{\approx} y^2x^2 \vdash x^2y^2 \approx y^2x^2. \]
Therefore, the inclusion $\bfV \subseteq \bfQ$ holds by Lemma~\ref{L: bases A B E Q}(iv).
Conversely, if $\bfV \subseteq \bfQ$, then~$\bfV$ satisfies the identity $x^2y^2 \approx y^2x^2$; but since~$\monAzero$ does not satisfy this identity, $\monAzero \notin \bfV$.

(ii) Suppose that $\monQ \notin \bfV$.
Then~$\bfV$ satisfies $\tau: x^2hxtx^2 \approx x^2htx^2$ by Lemma~\ref{L: exclusion A0 Q B0}(ii), so that \[ \bfV \models xhxtx \stackrel{\eqref{id: xxyx=xyx},\eqref{id: xyxx=xyx}}{\approx} x^2hxtx^2 \stackrel{\tau}{\approx} x^2htx^2 \stackrel{\eqref{id: xxyx=xyx},\eqref{id: xyxx=xyx}}{\approx} xhtx \vdash xhxtx \approx xhtx. \]
Therefore, the inclusion $\bfV \subseteq \bfAzero$ holds by Lemma~\ref{L: bases A B E Q}(i).
Conversely, if $\bfV \subseteq \bfAzero$, then~$\bfV$ satisfies the identity $xhxtx \approx xhtx$; but since~$\monQ$ does not satisfy this identity, $\monQ \notin \bfV$.

(iii) Suppose that $\bfAzero \vee \bfQ \subseteq \bfV$.
Let~$\bfU$ be any variety in~$\frakL(\bfV)$ such that $\bfU \notin \frakL(\bfAzero \vee \bfQ)$.
Then $\bfU \nsubseteq \bfQ$ and $\bfU \nsubseteq \bfAzero$, so that $\monAzero,\monQ \in \bfU$ by parts~(i) and~(ii).
Hence $\bfU \in [\bfAzero \vee \bfQ,\bfV]$.
\end{proof}

A word~$\bfu$ is written in \textit{natural form} if \[ \bfu = \bfu_0 \prod_{i=1}^m (h_i\bfu_i) \] where the variables $h_1,h_2,\ldots,h_m \in \scrA$ are precisely all simple variables of~$\bfu$ and the variables of $\bfu_0,\bfu_1,\ldots,\bfu_m \in \scrA^*$ are all non-simple in~$\bfu$.
Note that if all variables of~$\bfu$ are simple, then $\bfu = h_1h_2 \cdots h_m$; if all variables of~$\bfu$ are non-simple, then $\bfu = \bfu_0$.

\begin{lemma}[{Lee \cite[Lemma~5.1]{Lee14rm}}] \label{L: id Q}
Suppose that the words \[ \bfu = \bfu_0 \prod_{i=1}^m (h_i\bfu_i) \ \text{ and } \ \bfv = \bfv_0 \prod_{i=1}^n (t_i\bfv_i) \] are written in natural form.
Then $\bfQ \models \bfu\approx \bfv$ if and only if $m=n$\textup;\, $h_i = t_i$ for all $i \in \{ 1,2,\ldots, m\}$\textup;\, and $\content(\bfu_i) = \content(\bfv_i)$ for all $i \in \{ 0,1,\ldots,m\}$.
\end{lemma}

\begin{proposition} \label{P: A0 vee Q}
The variety $\bfAzero \vee \bfQ$ is Cross.
In particular\textup,
\begin{enumerate}[\rm(i)] 
\item the variety $\bfAzero \vee \bfQ$ is defined by the identities
\begin{equation}
\eqref{id: xxyx=xyx}, \quad \eqref{id: xyxx=xyx}, \quad xyhxty \approx yxhxty, \quad xhytxy \approx xhytyx; \label{id: basis A0 Q}
\end{equation}

\item the lattice $\frakL(\bfAzero \vee \bfQ)$ is given as follows\textup:
\[ \arraycolsep=3pt\def\arraystretch{1.25}
\begin{array}{ccccccrcccl}
           &         &            &         &            &         &             &         & \bfAzero & \subset & \bfAzero \vee \bfQ \\
           &         &            &         &            &         &             &         & \cup     &         & \cup \\
           &         &            &         &            &         & \bfE        & \subset & \bfBzero & \subset & \bfQ \\
           &         &            &         &            &         & \cup        &         & \cup     &         &      \\
\bftrivial & \subset & \bfRq\{1\} & \subset & \bfRq\{x\} & \subset & \bfRq\{xy\} & \subset & \bfEdual &         & 
\end{array}
\]
\end{enumerate}
\end{proposition}

\begin{proof}
(i) It is easily checked---either directly or by referring to Lemma~\ref{L: bases A B E Q}---that the monoids~$\monAzero$ and~$\monQ$ satisfy the identities~\eqref{id: basis A0 Q}.
It follows from O.B.\ Sapir \cite[Theorem~4.5]{SapO15a} that the variety $\bfAzero \vee \bfQ$ can be defined by $\eqref{id: basis A0 Q} \cup \Sigma$ for some set~$\Sigma$ of identities from~\eqref{id: O sub xhxhx}.
Consider any identity \[ \sigma: x^{e_0} \prod_{i=1}^m (h_ix^{e_i}) \approx x^{f_0} \prod_{i=1}^m (h_ix^{f_i}) \] from~$\Sigma$, where\, $e_0,f_0, e_1,f_1,\ldots,e_m,f_m \geq 0$;\, $\sum_{i=0}^me_i,\, \sum_{i=0}^mf_i \geq 2$;\, and $m \geq 0$.
It follows from Lemma~\ref{L: id Q} that for each $i \in \{ 0,1,\ldots,m \}$, either $e_i = f_i = 0$ or $e_i,f_i \geq 1$.
If $m \geq 1$, then the identity~$\sigma$ is clearly deducible from $\{ \eqref{id: xxyx=xyx}, \eqref{id: xyxx=xyx} \}$; if $m = 0$, then the identity~$\sigma$ is $x^{e_0} \approx x^{f_0}$ with $e_0,f_0 \geq 2$ and so is deducible from the consequence $x^3 \approx x^2$ of~\eqref{id: xxyx=xyx}. 
Therefore, the identities in~$\Sigma$ are not required in the definition of the variety $\bfAzero \vee \bfQ$, whence $\bfAzero \vee \bfQ = \var\{ \eqref{id: basis A0 Q} \}$.

(ii) If~$\bfV$ is any proper subvariety of $\bfAzero \vee \bfQ$, then either $\monAzero \notin \bfV$ or $\monQ \notin \bfV$, whence by Corollary~\ref{C: exclusion A0 Q}, either $\bfV \subseteq \bfQ$ or $\bfV \subseteq \bfAzero$.
Therefore, $\bfAzero$ and~$\bfQ$ are the only maximal subvarieties of $\bfAzero \vee \bfQ$.
Now the variety~$\bfBzero$ is the unique maximal subvariety of both~$\bfAzero$ and~$\bfQ$, and the lattice $\frakL(\bfBzero)$ is as given in the proposition \cite{Lee08,Lee14rm}.
\end{proof}

\section{Non-finitely generated almost Cross varieties} \label{sec: almost nfg}

This section presents almost Cross subvarieties of~$\bbJ$ that are non-finitely generated.
Results related to these varieties that are required later will also be established.

\subsection{The varieties \texorpdfstring{$\bfF$}{F} and \texorpdfstring{$\bfFdual$}{F-}} \label{subsec: var F}

Let~$\bfF$ denote the variety defined by the identities  \[ \eqref{id: xyxx=xyx}, \quad x^2hx \approx x^2h, \quad x^2y^2 \approx y^2x^2. \]
It is routinely checked that \[ \bfF = \bfJ_2 \{ \eqref{id: xyxx=xyx},\, x^2hx \approx x^2h \}. \]
Since $\monE \in \bfF \backslash \bfFdual$, the variety~$\bfF$ is not self-dual.

\begin{proposition}[{Gusev and Vernikov \cite[Proposition~6.1]{GusVer18}}] \label{P: F}
\begin{enumerate}[\rm(i)] 
\item The variety $\bfF$ is a non-finitely generated almost Cross variety.
\item The lattice $\frakL(\bfF)$ is isomorphic to the infinite chain $1<2<3<\cdots<\infty$.
\end{enumerate}
\end{proposition}

A description of every proper subvariety of~$\bfF$ is omitted since it is neither straightforward nor required in the present article.
But one subvariety of~$\bfF$ that is useful later is \[ \bfFsub = \bfF\{xyhxy \approx yxhxy\} = \var\{ \eqref{id: xyxx=xyx}, \; x^2hx \approx x^2h, \; xyhxy \approx yxhxy\}, \] whose subvarieties constitute a lower region of the lattice $\frakL(\bfF)$: \[ \bftrivial \subset \bfRq\{1\} \subset \bfRq\{x\} \subset \bfRq\{xy\} \subset \bfE \subset \bfFsub \subset \cdots \subset \bfF; \] see Gusev and Vernikov \cite[Chapter~6.1]{GusVer18} for more information on all subvarieties of~$\bfF$.
In particular, $\bfE$ is the unique maximal subvariety of~$\bfFsub$.
It follows that~$\bfFsub$ is generated by any one of its monoids that does not belong to~$\bfE$, for instance, \begin{align*} \monFsub & = \langle a,b,1 \,|\, a^2=b^2a=0,\, ab=a,\, b^3 = b^2\rangle \\ & = \{ 0,\, a,\, b,\, ba,\, b^2,\, 1\}. \end{align*}

\begin{lemma} \label{L: excl Fsub Edual}
Let~$\bfV$ be any subvariety of~$\bfJ_2\{ \eqref{id: xyxx=xyx} \}$.
Then
\begin{enumerate}[\rm(i)]
\item $\monFsub \notin \bfV$ if and only if $\bfV \models \eqref{id: xxyx=xyx}$\textup;
\item $\bfF \nsubseteq \bfV$ and $\monEdual \notin \bfV$ imply that~$\bfV$ is Cross.
\end{enumerate}
\end{lemma}

\begin{proof}
(i) If~$\bfV$ is completely regular, then $\bfV \subseteq \bfJ_1$ by Lemma~\ref{L: comm CR}(i), so that both $\monFsub \notin \bfV$ and $\bfV \models \eqref{id: xxyx=xyx}$ hold.
If~$\bfV$ is not completely regular, then the result follows from Gusev \cite[Lemma~2.10]{Gus25ijac}.
 
(ii) Suppose that $\bfF \nsubseteq \bfV$ and $\monEdual \notin \bfV$.
Then since $\bfEdual \subseteq \bfBzero$ by Proposition~\ref{P: A0 vee Q}(ii), it follows from the assumption $\monEdual \notin \bfV$ that $\monBzero \notin \bfV$.
Therefore, by Lemma~\ref{L: exclusion A0 Q B0}(iii), either $\bfV \models x^2 \approx x$, $\bfV \models x^2hx^2 \approx x^2h$, or $\bfV \models x^2hx^2 \approx hx^2$.

\medskip

\noindent\textsc{Case~1}: $\bfV \models x^2 \approx x$. 
Then~$\bfV$ is completely regular and so is Cross by Lemma~\ref{L: comm CR}.

\medskip

\noindent\textsc{Case~2}: $\bfV$ satisfies $\sigma: x^2hx^2 \approx x^2h$.
Then since \[ \bfV \models x^2h \stackrel{\sigma}{\approx} x^2hx^2 \stackrel{\eqref{id: xyxx=xyx}}{\approx} x^2hx \vdash x^2hx \approx x^2h, \] the inclusion $\bfV \subseteq \bfJ_2 \{ \eqref{id: xyxx=xyx},\, x^2hx \approx x^2h \} = \bfF$ holds. 
But this inclusion is proper due to the assumption $\bfV \neq \bfF$.
It follows that~$\bfV$ is Cross since~$\bfF$ is almost Cross by Proposition~\ref{P: F}(i).

\medskip

\noindent\textsc{Case~3}: $\bfV$ satisfies $\tau: x^2hx^2 \approx hx^2$.
Then since \[ \bfV \models hx^2 \stackrel{\tau}{\approx} x^2hx^2 \stackrel{\text{\ref{id: aperiodic}}_2}{\approx} xx^2hx^2 \stackrel{\tau}{\approx} xhx^2 \stackrel{\eqref{id: xyxx=xyx}}{\approx} xhx \vdash xhx \approx hx^2, \] it follows from Lemma~\ref{L: xyx not isoterm}(ii) that~$\bfV$ is Cross.
\end{proof}

\subsection{The varieties \texorpdfstring{$\bfH$}{H} and \texorpdfstring{$\bfHdual$}{H-}} \label{subsec: var H}

Let~$\bfH$ denote the variety defined by the identities $\text{\ref{id: eventually}}_2$, \eqref{id: xxyx=xyx}, \eqref{id: xyxx=xyx}, and \begin{equation} xhxy^2x \approx xhy^2x; \label{id: basis H} \end{equation} in other words, \[ \bfH = \bfJ_2 \{ \eqref{id: xxyx=xyx},\, \eqref{id: xyxx=xyx},\, \eqref{id: basis H} \}. \]
It is routinely checked that $\bfH \neq \bfHdual$ and that the monoids~$\monAzero$ and~$\monQ$ belong to both~$\bfH$ and~$\bfHdual$.

The main goal of this subsection is to show that the variety~$\bfH$ is almost Cross and to provide a complete description of its lattice $\frakL(\bfH)$ of subvarieties.
Since \[ \frakL(\bfH) = \frakL(\bfAzero \vee \bfQ) \cup [\bfAzero \vee \bfQ,\,\bfH] \] by Corollary~\ref{C: exclusion A0 Q}(iii) and a description of the lattice $\frakL(\bfAzero \vee \bfQ)$ can be found in Proposition~\ref{P: A0 vee Q}(ii), it suffices to describe the interval $[\bfAzero \vee \bfQ,\,\bfH]$.

Recall that $\text{\ref{id: Hsub}}_m$ denotes the restrictive identity \[ xy \prod_{i=1}^m (t_i\bfa_i) \approx yx \prod_{i=1}^m (t_i\bfa_i), \] where $(\bfa_1,\bfa_2,\bfa_3,\bfa_4,\ldots)$ is the alternating sequence $(x,y,x,y,\ldots)$.
For each $m \geq 2$, define \[ \bfH_m = \bfH\{ \text{\ref{id: Hsub}}_m \}. \]

\begin{lemma} \label{L: H2}
$\bfH_2 = \bfAzero \vee \bfQ$.
\end{lemma}

\begin{proof}
It is easily checked that $\monAzero$ and $\monQ$ satisfy the identities $\{ \text{\ref{id: eventually}}_2, \eqref{id: xxyx=xyx}, \eqref{id: xyxx=xyx}, \eqref{id: basis H}, \text{\ref{id: Hsub}}_2 \}$ that define~$\bfH_2$, so the inclusion $\bfAzero \vee \bfQ \subseteq \bfH_2$ holds.
To establish the reverse inclusion, it suffices to show that~$\bfH_2$ satisfies the four identities from~\eqref{id: basis A0 Q} that define $\bfAzero \vee \bfQ$; see Proposition~\ref{P: A0 vee Q}(i).
The first three identities $\{\eqref{id: xxyx=xyx}, \eqref{id: xyxx=xyx}, \text{\ref{id: Hsub}}_2 \}$ from~\eqref{id: basis A0 Q} are clearly satisfied by~$\bfH_2$.
The fourth identity $xhytxy \approx xhytyx$ from~\eqref{id: basis A0 Q} is also satisfied by~$\bfH_2$ because
\begin{align*}
\bfH \models xhytxy & \stackrel{\eqref{id: xyxx=xyx}}{\approx} xhytxx^2y \stackrel{\eqref{id: basis H}}{\approx} xhytxyx^2y \stackrel{\eqref{id: xyxx=xyx}}{\approx} xhytxyxy \\ & \stackrel{\text{\ref{id: eventually}}_2}{\approx} xhytyxyx \stackrel{\eqref{id: xyxx=xyx}}{\approx} xhytyx^2y^2x \stackrel{\eqref{id: basis H}}{\approx} xhytx^2y^2x \\ & \stackrel{\eqref{id: basis H}}{\approx} xhyty^2x \stackrel{\eqref{id: xyxx=xyx}}{\approx} xhytyx \vdash xhytxy \approx xhytyx. \qedhere
\end{align*}
\end{proof}

\begin{lemma} \label{L: H other id}
The variety~$\bfH$ satisfies the identities
\begin{align}
xy^2x & \approx xyxy, \label{id: H xyyx=xyxy} \\
xy^2x & \approx yxyx, \label{id: H xyyx=yxyx} \\
xyhxy & \approx yxhxy, \label{id: H xyhxy=yxhxy} \\
xhxyxty & \approx xhyxty, \label{id: H O1} \\ 
xhytxyx & \approx xhytyx. \label{id: H O2}
\end{align} 
\end{lemma}

\begin{proof}
The variety~$\bfH$ satisfies the identities \eqref{id: H xyhxy=yxhxy}--\eqref{id: H O2} because
\begin{align*}
\bfH \models xyhxy & \stackrel{\eqref{id: xxyx=xyx}}{\approx} (xy)^2hxy \stackrel{\text{\ref{id: eventually}}_2}{\approx} (yx)^2hxy \stackrel{\eqref{id: xyxx=xyx}}{\approx} (yx)^2hx^2y \stackrel{\eqref{id: basis H}}{\approx} (yx)^2hyx^2y \\
& \stackrel{\eqref{id: xxyx=xyx}}{\approx} yxhyx^2y \stackrel{\eqref{id: basis H}}{\approx} yxhx^2y \stackrel{\eqref{id: xyxx=xyx}}{\approx} yxhxy \vdash \eqref{id: H xyhxy=yxhxy}; \\
\bfH \models xhxyxty & \stackrel{\eqref{id: xxyx=xyx}}{\approx} xhxy^2xty \stackrel{\eqref{id: basis H}}{\approx} xhy^2xty \stackrel{\eqref{id: xxyx=xyx}}{\approx} xhyxty \vdash \eqref{id: H O1}; \\
\bfH \models xhytxyx & \stackrel{\eqref{id: xyxx=xyx}}{\approx} xhytxy^2x \stackrel{\eqref{id: basis H}}{\approx} xhyty^2x \stackrel{\eqref{id: xyxx=xyx}}\approx xhytyx \vdash \eqref{id: H O2}.
\end{align*}
The deduction $\eqref{id: H xyhxy=yxhxy} \vdash \eqref{id: H xyyx=yxyx}$ holds vacuously, so that $\bfH \models \eqref{id: H xyyx=yxyx}$.
Hence \[ \bfH \models xyxy \stackrel{\text{\ref{id: eventually}}_2}{\approx} yxyx \stackrel{\eqref{id: H xyyx=yxyx}}{\approx} xy^2x \vdash \eqref{id: H xyyx=xyxy}. \qedhere \]
\end{proof}

\begin{lemma} \label{L: [A01veeQ,H]}
$[\bfAzero \vee \bfQ,\, \bfH] = \{ \bfH_m \,|\, m \geq 2\}$.
\end{lemma} 

\begin{proof}
It is easily seen that $\bfH_m \in [\bfAzero \vee \bfQ,\, \bfH]$ for all $m \geq 2$.
Conversely, let~$\bfV$ be any variety in $[\bfAzero \vee \bfQ,\, \bfH]$.
By Lemma~\ref{L: H other id}, the variety~$\bfH$ satisfies the identities $\{ \eqref{id: H O1},\eqref{id: H O2} \} = \{ \eqref{id: O} \}$ and so is a subvariety of $\bfO \cap \bfJ_2$.
Hence it follows from Proposition~\ref{P: subvarieties of O} that $\bfV = \bfH(\Sigma_1 \cup \Sigma_2)$ for some $\Sigma_1 \subseteq \eqref{id: Oc sub}$ and some $\Sigma_2 \subseteq \eqref{id: O J2 sub yhpc=yhqc}$.
Since $\bfAzero \vee \bfQ \subseteq \bfV = \bfH(\Sigma_1 \cup \Sigma_2)$, 
\begin{enumerate}[\ (a)]
\item $\monAzero$ and $\monQ$ satisfy every identity in $\Sigma_1 \cup \Sigma_2$.
\end{enumerate}
On the other hand, any identity in $\Sigma_1 \cup \Sigma_2$ that is satisfied by~$\bfH$ is redundant in the definition of $\bfV = \bfH(\Sigma_1 \cup \Sigma_2)$.
Hence it can further be assumed that
\begin{enumerate}[\ (a)]
\item[(b)] $\bfH$ does not satisfy any identity in $\Sigma_1 \cup \Sigma_2$.
\end{enumerate}

It is easily checked that each identity from~\eqref{id: Oc sub} is either not satisfied by~$\monQ$ or satisfied by~$\bfH$.
Therefore, by~(a) and~(b), none of the identities from~\eqref{id: Oc sub} belongs to~$\Sigma_1$. 
Hence $\Sigma_1$ is empty and so $\bfV = \bfH\Sigma_2$.
Consider any identity from $\Sigma_2 \subseteq \eqref{id: O J2 sub yhpc=yhqc}$, say \[ \sigma: \bfh\bfp\bfc \approx \bfh\bfq\bfc, \] where $\bfh \in \{ 1,\, yh\}$, $\bfp,\bfq \in \{x,y\}^+$, and $\bfc \in \big\{ 1,\, txy,\, \prod_{i=1}^m (t_i\bfa_i), \,\prod_{i=2}^{m+1} (t_i\bfa_i) \,\big|\, m \geq 1 \big\}$ with \[ 1 \leq |\bfp|_x = |\bfq|_x,\, |\bfp|_y = |\bfq|_y \leq 2; \quad |\bfh\bfp\bfc|_x = |\bfh\bfq\bfc|_x,\, |\bfh\bfp\bfc|_y = |\bfh\bfq\bfc|_y \geq 2; \] and $(\bfa_1,\bfa_2,\bfa_3,\bfa_4,\ldots) = (x,y,x,y,\ldots)$.
In the following, it is shown that $\bfH\{\sigma\} = \bfH_m$ for some $m \geq 2$.
The proof is then complete since the inclusions $\bfH_2 \subseteq \bfH_3 \subseteq \bfH_4 \subseteq \cdots$ are easily checked.

Now according to the above conditions, the words~$\bfp$ and~$\bfq$ can be any of \[ x^ry^s, \quad y^rx^s, \quad xy^rx, \quad yx^ry, \quad xyxy, \quad yxyx, \quad r,s \in \{1,2\}. \]
Since~$x$ and~$y$ are non-simple variables of both $\bfh\bfp\bfc$ and $\bfh\bfq\bfc$, the identities $\{ \eqref{id: xxyx=xyx}, \eqref{id: xyxx=xyx} \}$ can be used to convert~$\sigma$ into $\bfh\bfp'\bfc \approx \bfh\bfq'\bfc$, where~$\bfp'$ and~$\bfq'$ can be any of \[ x^2y^2, \quad y^2x^2, \quad xy^2x, \quad yx^2y, \quad xyxy, \quad yxyx. \]
The identities $\{ \text{\ref{id: eventually}}_2,\eqref{id: H xyyx=xyxy}, \eqref{id: H xyyx=yxyx} \}$ can then be used to convert $\bfh\bfp'\bfc \approx \bfh\bfq'\bfc$ into $\bfh\bfp''\bfc \approx \bfh\bfq''\bfc$, where $\bfp'',\bfq'' \in \{ x^2y^2,\, y^2x^2,\, yxyx \}$.
In summary, the identities $\{ \eqref{id: xxyx=xyx}, \eqref{id: xyxx=xyx}, \text{\ref{id: eventually}}_2,\eqref{id: H xyyx=xyxy}, \eqref{id: H xyyx=yxyx} \}$ of~$\bfH$ can be used to convert~$\sigma$ into $\bfh\bfp''\bfc \approx \bfh\bfq''\bfc$.
Therefore $\bfH\{ \sigma \} = \bfH\{\bfh\bfp''\bfc \approx \bfh\bfq''\bfc\}$, whence one may assume that
\begin{enumerate}[\ (a)]
\item[(c)] $\bfp,\bfq \in \{ x^2y^2,\, y^2x^2,\, yxyx \}$.
\end{enumerate}

Suppose that $\bfh = yh$.
Then by~(b) and~(c), the identity $\sigma$ can be any of the following: \[ \sigma_1: yh \cdot x^2y^2\bfc \approx yh \cdot y^2x^2\bfc, \quad \sigma_2: yh \cdot x^2y^2\bfc \approx yh \cdot yxyx\bfc, \quad \sigma_3: yh \cdot y^2x^2\bfc \approx yh \cdot yxyx\bfc. \]
Since $\bfH \models \{ \eqref{id: xyxx=xyx}, \eqref{id: basis H}, \eqref{id: H xyyx=xyxy} \}$ by Lemma~\ref{L: H other id} and that the deductions \[ yh \cdot x^2y^2\bfc \stackrel{\eqref{id: basis H}}{\approx} yh \cdot yx^2y^2\bfc \stackrel{\eqref{id: xyxx=xyx}}{\approx} yh \cdot yx^2y\bfc \stackrel{\eqref{id: H xyyx=xyxy}}{\approx} yh \cdot yxyx\bfc \] hold, it follows that $\bfH \models \sigma_2$ and $\bfH\{\sigma_1\} = \bfH\{\sigma_3\}$.
Hence, $\sigma$ cannot be~$\sigma_2$ due to~(b), and it suffices to assume that~$\sigma$ is~$\sigma_3$.
In fact, $\sigma$ can be assumed to be the simpler identity $\sigma_3':y^2x^2\bfc \approx yxyx\bfc$, obtained by removing the prefix $\bfh = yh$ from both sides of~$\sigma_3$.
Indeed, the inclusion $\bfH\{\sigma_3'\} \subseteq \bfH\{\sigma_3\}$ clearly holds, while the reverse inclusion also holds because \[ \bfH\{\sigma_3\} \models y^2x^2\bfc \stackrel{\text{\ref{id: aperiodic}}_2}{\approx} yy \cdot y^2x^2\bfc \stackrel{\sigma_3}{\approx} yy \cdot yxyx\bfc \stackrel{\eqref{id: xxyx=xyx}}{\approx} yxyx\bfc \vdash \sigma_3', \] so that $\bfH\{\sigma\} = \bfH\{\sigma_3\} = \bfH\{ \sigma_3'\}$.
Consequently, one may assume that $\bfh=1$.

The identity~$\sigma$ is thus $\bfp\bfc \approx \bfq\bfc$.
Then it follows from~(b) and~(c) that $\bfH\{\sigma\}$ can be any of the following varieties: \[ \bfH\{ x^2y^2\bfc \approx y^2x^2\bfc\}, \quad \bfH\{ x^2y^2\bfc \approx yxyx\bfc\}, \quad \bfH\{ y^2x^2\bfc \approx yxyx\bfc\}; \] but since $\bfH \models xyxy \approx yxyx$, the second and third varieties coincide.
Therefore, it remains to assume that
\begin{enumerate}[\ (a)]
\item[(d)] $\sigma$ is either $x^2y^2\bfc \approx y^2x^2\bfc$ or $y^2x^2\bfc \approx yxyx\bfc$.
\end{enumerate}
It is routinely checked that if $\bfc=1$, then $\monAzero \not\models \sigma$; and if $\bfc = txy$, then the deduction $\eqref{id: H xyhxy=yxhxy} \vdash \sigma$ holds, so that $\bfH \models \sigma$ by Lemma~\ref{L: H other id}.
But these contradict~(a) and~(b).
Hence $\bfc \notin \{1,txy\}$ and so $\bfc \in \big\{ \prod_{i=1}^m (t_i\bfa_i), \,\prod_{i=2}^{m+1} (t_i\bfa_i) \,\big|\, m \geq 1 \big\}$.
Equivalently,
\begin{enumerate}[\ (a)]
\item[(e)] either $\bfc = \prod_{i=1}^m (t_i\bfa_i)$ for some $m \geq 1$ or $\bfc = \prod_{i=2}^m (t_i\bfa_i)$ for some $m \geq 2$.
\end{enumerate}
There are three cases to consider.

\medskip

\noindent\textsc{Case~1}: $\sigma$ is $x^2y^2\bfc \approx y^2x^2\bfc$.
If $y \notin \content(\bfc)$, then $\bfc = t_1x$ and it is straightforwardly checked that $\monAzero \not\models \sigma$, contradicting~(a).
A similar contradiction is obtained if $x \notin \content(\bfc)$.
Therefore, $x,y \in \content(\bfc)$, so that~$\bfc$ is either $\prod_{i=1}^m (t_i\bfa_i)$ or $\prod_{i=2}^{m+1} (t_i\bfa_i)$ for some $m \geq 2$.
Then the equivalence $\{ \eqref{id: xxyx=xyx},\, x^2y^2 \bfc \approx y^2x^2 \bfc\} \sim \{ \eqref{id: xxyx=xyx},\, xy \bfc \approx yx \bfc\}$ holds, so that \[ \bfH\{\sigma\} = \bfH\{x^2y^2\bfc \approx y^2x^2\bfc\} = \bfH\{xy\bfc \approx yx\bfc\} = \bfH\{ \text{\ref{id: Hsub}}_m \} = \bfH_m. \]

\noindent\textsc{Case~2}: $\sigma$ is $y^2x^2 \prod_{i=2}^m (t_i\bfa_i) \approx yxyx \prod_{i=2}^m (t_i\bfa_i)$ for some $m \geq 2$.
Then the inclusion $\bfH\{\sigma\} \subseteq \bfH\{\text{\ref{id: Hsub}}_m\}$ holds because 
\begin{align*} 
\bfH\{\sigma\} \models xy \prod_{i=1}^m (t_i\bfa_i) & \stackrel{\eqref{id: xxyx=xyx}}{\approx} x^2y^2 \prod_{i=1}^m (t_i\bfa_i) \stackrel{\sigma}{\approx} xyxy \prod_{i=1}^m (t_i\bfa_i) 
\stackrel{\text{\ref{id: eventually}}_2}{\approx} yxyx \prod_{i=1}^m (t_i\bfa_i) \\ & \stackrel{\sigma}{\approx} y^2x^2 \prod_{i=1}^m (t_i\bfa_i) \stackrel{\eqref{id: xxyx=xyx}}{\approx} yx \prod_{i=1}^m (t_i\bfa_i) \vdash \text{\ref{id: Hsub}}_m, \end{align*} 
while the reverse inclusion $\bfH\{\sigma\} \supseteq \bfH\{\text{\ref{id: Hsub}}_m\}$ holds because
\[ \bfH\{\text{\ref{id: Hsub}}_m\} \models \text{\ref{id: Hsub}}_m \vdash xyx\prod_{i=2}^m (t_i\bfa_i) \approx yx^2\prod_{i=2}^m (t_i\bfa_i) \vdash \sigma. \]
Consequently, $\bfH\{\sigma\} = \bfH\{\text{\ref{id: Hsub}}_m\} = \bfH_m$.

\medskip

\noindent\textsc{Case~3}: $\sigma$ is $y^2x^2 \prod_{i=1}^m (t_i\bfa_i) \approx yxyx \prod_{i=1}^m (t_i\bfa_i)$ for some $m \geq 1$.
If $m = 1$, then $\sigma$ is $y^2x^2 t_1 x \approx yxyx t_1 x$ and it is easily checked that $\monAzero \not\models \sigma$, contradicting~(a).
Hence $m \geq 2$, so that $\prod_{i=1}^m (t_i\bfa_i) = t_1xt_2y \cdots$.
Consider the identity $\sigma' : y^2x^2 \prod_{i=2}^m (t_i\bfa_i) \approx yxyx \prod_{i=2}^m (t_i\bfa_i)$ obtained by removing the factor $t_1\bfa_1$ from both sides of~$\sigma$.
The inclusion $\bfH\{\sigma\} \subseteq \bfH\{\sigma'\}$ holds because \[ \bfH\{\sigma\} \models y^2x^2 \prod_{i=2}^m (t_i\bfa_i) \stackrel{\text{\ref{id: aperiodic}}_2}{\approx} y^2x^2 \cdot x\prod_{i=2}^m (t_i\bfa_i) \stackrel{\sigma}{\approx} yxyx \cdot x\prod_{i=2}^m (t_i\bfa_i) \stackrel{\text{\ref{id: aperiodic}}_2}{\approx} yxyx \prod_{i=2}^m (t_i\bfa_i) \vdash \sigma', \] while the reverse inclusion $\bfH\{\sigma\} \supseteq \bfH\{\sigma'\}$ holds because~$\sigma$ is obtained by making the substitution $t_2 \mapsto t_1xt_2$ in~$\sigma'$.
It follows that $\bfH\{\sigma\} = \bfH\{\sigma'\} = \bfH_m$ by Case~2.

\medskip

Since the three cases cover all scenarios from~(d) and~(e), the proof is complete.
\end{proof}

\begin{proposition} \label{P: H}
\begin{enumerate}[\rm(i)] 
\item The variety $\bfH$ is a non-finitely generated almost Cross variety.
\item The lattice $\frakL(\bfH)$ is given in Figure~\ref{F: H}.
\end{enumerate}
\end{proposition}

\begin{proof}
As observed earlier, the lattice $\frakL(\bfH)$ coincides with $\frakL(\bfAzero \vee \bfQ) \cup [\bfAzero \vee \bfQ,\,\bfH]$.
It is easily seen that the deductions $\text{\ref{id: Hsub}}_2 \vdash \text{\ref{id: Hsub}}_3 \vdash \text{\ref{id: Hsub}}_4 \vdash \cdots$ hold, so it follows from Lemmas~\ref{L: H2} and~\ref{L: [A01veeQ,H]} that the interval $[\bfAzero \vee \bfQ,\,\bfH]$ coincides with the chain \begin{equation} \bfAzero \vee \bfQ = \bfH_2 \subseteq \bfHsub \subseteq \cdots \subseteq \bfH. \label{D: H} \end{equation}
It is easily checked that for any $k > m$, the identities $\{ \text{\ref{id: aperiodic}}_2,\, \text{\ref{id: eventually}}_2,\, \eqref{id: xxyx=xyx},\, \eqref{id: xyxx=xyx},\, \eqref{id: basis H},\, \text{\ref{id: Hsub}}_k \}$ defining~$\bfH_k$ can only convert $xy \prod_{i=1}^m (t_i\bfa_i)$ into a word of the form $x^py^q \prod_{i=1}^m (t_i\bfa_i^{r_i})$, where $p,q,r_1,r_2,\ldots,r_m \geq 1$.
Hence $\bfH_k \not\models \text{\ref{id: Hsub}}_m$, so that $\bfH_k \neq \bfH_m$.
It follows that the inclusions in~\eqref{D: H} are all proper.
By Proposition~\ref{P: A0 vee Q}(ii), the lattice $\frakL(\bfH)$ is as given in Figure~\ref{F: H}.

The variety~$\bfH$ satisfies the identities $\{ \eqref{id: xxyx=xyx}, \eqref{id: xyxx=xyx} \}$ and so is locally finite by Lemma~\ref{L: LF FG max}(i).
It is clear that every proper subvariety of~$\bfH$ is finitely based and small, and so also finitely generated by Lemma~\ref{L: LF FG max}(ii).
Consequently, $\bfH$ is almost Cross.
Since~$\bfH$ has no maximal subvarieties, it is non-finitely generated by Lemma~\ref{L: LF FG max}(iii).
\end{proof}

\begin{figure}[ht]
\begin{center}
\begin{tikzpicture}[scale=0.4]
\draw [fill] (03,21) circle [radius=\dt];
\draw [fill] (03,18) circle [radius=\dt];
\draw [fill] (03,16) circle [radius=\dt]; 
\draw [fill] (03,14) circle [radius=\dt];
\draw [fill] (00,12) circle [radius=\dt]; \draw [fill] (06,12) circle [radius=\dt]; 
\draw [fill] (03,10) circle [radius=\dt];
\draw [fill] (00,08) circle [radius=\dt]; \draw [fill] (06,08) circle [radius=\dt];
\draw [fill] (03,06) circle [radius=\dt];
\draw [fill] (03,04) circle [radius=\dt];
\draw [fill] (03,02) circle [radius=\dt];
\draw [fill] (03,00) circle [radius=\dt];
\draw (03,00) -- (03,06) -- (00,08) -- (06,12) -- (03,14) -- (03,19) (03,06) -- (06,08) -- (00,12) -- (03,14);
\node [right] at (03.2,21.1) {$\bfH$};
\node at (03,20.2) {$\vdots$};
\node [right] at (03.2,18) {$\bfH_4$};
\node [right] at (03.2,16) {$\bfHsub$};
\node [right] at (03.2,14.1) {$\bfH_2$};
\node [left] at (-0.2,12) {$\bfAzero$}; \node [right] at (06.2,12) {$\bfQ$};
\node [right] at (03.3,9.95) {$\bfBzero$};
\node [left] at (-0.2,08) {$\bfE$}; \node [right] at (06.1,08.15) {$\bfEdual$};
\node [right] at (03.2,05.8) {$\bfRq\{xy\}$};
\node [right] at (03.2,03.95) {$\bfRq\{x\}$};
\node [right] at (03.2,01.95) {$\bfRq\{1\}$};
\node [right] at (03.2,00) {$\bftrivial$};
\end{tikzpicture}
\end{center} \caption{The lattice $\frakL(\bfH)$}
\label{F: H}
\end{figure}

\begin{remark} \label{R: Hm}
\begin{enumerate}[(i)]
\item The variety $\bfH_2 = \bfAzero \vee \bfQ$ is self-dual but the varieties $\bfHsub, \bfH_4, \ldots, \bfH$ are not. 
For instance, the following monoid belongs to $\bfHsub \subset \bfH_4 \subset \cdots \subset \bfH$ but not to~$\overleftarrow{\bfH}$:
\begin{align*}
\monHsub & = \langle a,e,f,1 \,|\, ae=fa=a,\, af=fe=0,\, e^2=e,\, f^2=f \rangle \\ & = \{ 0,\, a,\, e,\, f,\, ea,\, ef,\, 1\}.
\end{align*}
\item Since $\monHsub \in \bfHsub$ and $\monHsub \notin \bfH_2$, the variety~$\bfHsub$ is generated by~$\monHsub$. 
\end{enumerate}
\end{remark}

\subsection{The varieties \texorpdfstring{$\bfL$}{L}, \texorpdfstring{$\bfP$}{P}, and \texorpdfstring{$\bfPdual$}{P-}} \label{subsec: var L P}

Let~$\bfL$ denote the variety defined by the identities \[ \text{\ref{id: aperiodic}}_2, \quad xyhxty \approx yxhxty, \quad xhxyty \approx xhyxty, \quad xhytxy \approx xhytyx \] and let~$\bfP$ denote the variety defined by the identities \[ \eqref{id: xyxx=xyx}, \quad x^2y^2 \approx y^2x^2, \quad xyhxy \approx yxhxy, \quad xyxhy \approx x^2yhy, \quad xyxhy \approx yx^2hy. \]
It is easily seen that~$\bfL$ is self-dual, while~$\bfP$ is not self-dual because $\monFsub \in \bfP \backslash \bfPdual$.

\begin{proposition}[{Lee \cite[Proposition~4.1]{Lee14zns}}] \label{P: L}
\begin{enumerate}[\rm(i)] 
\item The variety $\bfL$ is a non-finitely generated almost Cross variety.
\item The lattice $\frakL(\bfL)$ coincides with the chain \[ \bftrivial \subset \bfRq\{1\} \subset \bfRq\{x\} \subset \bfRq\{xy\} \subset \bfRq\{\bfx_1\} \subset \bfRq\{\bfx_2\} \subset \bfRq\{\bfx_3\} \subset \cdots \subset \bfL, \] where $\bfx_m = x \prod_{i=1}^m (h_ix)$.
\end{enumerate}
\end{proposition}

\begin{proposition}[{Gusev \cite[Proposition~3.1]{Gus25ijac}}] \label{P: P}
\begin{enumerate}[\rm(i)] 
\item The variety $\bfP$ is a non-finitely generated almost Cross variety.
\item The lattice $\frakL(\bfP)$ is given in Figure~\ref{F: P}\textup; the varieties in this lattice that have not been introduced are 
\begin{align*} 
\bfP_m & = \bfP\{\text{\ref{id: Hsub}}_m\}, \,\ m \geq 2, \\
\bfPsub & = \bfP\{xh_1xh_2xh_3x \approx xh_1xh_2h_3x\}, \\ 
\text{and } \ \bfFsub \vee \bfEdual & = \bfP_2 \cap \bfPsub = \bfP\{ \text{\ref{id: Hsub}}_2,\, xh_1xh_2xh_3x \approx xh_1xh_2h_3x\}.
\end{align*}
\end{enumerate}
\end{proposition}

\begin{figure}[ht]
\begin{center}
\begin{tikzpicture}[scale=0.4]
\draw [fill] (03,23) circle [radius=\dt];
\draw [fill] (03,20) circle [radius=\dt];
\draw [fill] (03,18) circle [radius=\dt];
\draw [fill] (03,16) circle [radius=\dt]; 
\draw [fill] (00,14) circle [radius=\dt]; \draw [fill] (06,14) circle [radius=\dt];
\draw [fill] (03,12) circle [radius=\dt]; \draw [fill] (09,12) circle [radius=\dt]; 
\draw [fill] (00,10) circle [radius=\dt]; \draw [fill] (06,10) circle [radius=\dt];
\draw [fill] (03,08) circle [radius=\dt]; \draw [fill] (09,08) circle [radius=\dt];
\draw [fill] (06,06) circle [radius=\dt];
\draw [fill] (06,04) circle [radius=\dt];
\draw [fill] (06,02) circle [radius=\dt];
\draw [fill] (06,00) circle [radius=\dt];
\draw 
(06,00) -- (06,06) -- (03,08) -- (09,12) -- (03,16) -- (03,21) 
(06,06) -- (09,08) -- (00,14) -- (03,16)
(03,08) -- (00,10) -- (06,14)
;
\node [right] at (03.2,23.1) {$\bfP$};
\node at (03,22.2) {$\vdots$};
\node [right] at (03.2,20) {$\bfP_5$};
\node [right] at (03.2,18) {$\bfP_4$};
\node [right] at (03.2,16.1) {$\bfP_3$};
\node [left] at (-0.1,14) {$\bfPsub$}; \node [right] at (06.2,14.1) {$\bfP_2$}; 
\node [left] at (2.6,12) {$\bfFsub \vee \bfEdual$}; \node [right] at (09.2,12) {$\bfQ$};
\node [left] at (-0.1,10) {$\bfFsub$}; \node [right] at (06.3,9.95) {$\bfBzero$};
\node [left] at (2.8,07.9) {$\bfE$}; \node [right] at (09.1,08.15) {$\bfEdual$};
\node [right] at (06.2,05.8) {$\bfRq\{xy\}$};
\node [right] at (06.2,03.95) {$\bfRq\{x\}$};
\node [right] at (06.2,01.95) {$\bfRq\{1\}$};
\node [right] at (06.2,00) {$\bftrivial$};
\end{tikzpicture}
\end{center} 
\caption{The lattice $\frakL(\bfP)$}
\label{F: P}
\end{figure}

\begin{lemma}[Gusev {\cite[Proposition~3.1 and Lemma~3.3]{Gus25ijac}}] \label{L: excl P}
Let $\bfV$ be any variety in the interval $[\bfFsub \vee \bfQ,\, \bfJ_2\{\eqref{id: xyxx=xyx}\}]$.
Suppose that $\bfP \nsubseteq \bfV$.
Then $\bfV \models \text{\ref{id: Hsub}}_m$ for some $m \geq 2$.
\end{lemma}


\section{Finitely generated almost Cross varieties} \label{sec: almost fg}

This section presents almost Cross subvarieties of~$\bbJ$ that are finitely generated.
Results related to these varieties that are required later will also be established.

\subsection{The varieties \texorpdfstring{$\bfI$}{I} and \texorpdfstring{$\bfIdual$}{I-}} \label{subsec: var I}

Let $\Pi_m$ denote the set of permutations on the set $\{ 1,2,\ldots,m\}$ and let~$\bfI$ denote the variety defined by the identities
\begin{align}
\eqref{id: xyxx=xyx}, \quad x^2y^2 \approx y^2x^2, \quad & xyhxy \approx yxhxy, \quad xh_1xh_2xh_3x \approx xh_1xh_2h_3x, \label{id: I first4} \\
x \bigg(\prod_{i=1}^m y_{\pi(i)}\bigg)x \bigg(\prod_{i=1}^m (h_iy_i)\bigg) & \approx x^2 \bigg(\prod_{i=1}^m y_{\pi(i)}\bigg)\bigg(\prod_{i=1}^m (h_iy_i)\bigg), \quad m \geq 1, \, \pi \in \Pi_m. \label{id: I list}
\end{align}
It is routinely checked that $\monFsub \models \eqref{id: I first4}$ and $\monFsub \models xhxyty \approx xhyxty \vdash \eqref{id: I list}$, so that $\monFsub \in \bfI$.
But since $\monFsub$ does not satisfy the identity~\eqref{id: xxyx=xyx} of~$\bfIdual$, the variety~$\bfI$ is not self-dual.

\begin{proposition}[{Gusev~\cite{Gus20}}] \label{P: I}
\begin{enumerate}[\rm(i)] 
\item The variety $\bfI$ is a non-finitely based\textup, finitely generated almost Cross variety.
\item The lattice $\frakL(\bfI)$ is given in Figure~\ref{F: I}\textup; the variety in this lattice that has not been introduced is \[ \bfIsub = \bfI\{xhxyty \approx xhyxty\} = \var\{ \eqref{id: I first4},\, xhxyty \approx xhyxty \}. \]
\end{enumerate}
\end{proposition}

\begin{remark}
O.B.\ Sapir~\cite{SapO21} exhibited a monoid of order~31 that generates the variety~$\bfI$.
\end{remark}

\begin{figure}[ht]
\begin{center}
\begin{tikzpicture}[scale=0.4]
\draw [fill] (03,18) circle [radius=\dt];
\draw [fill] (03,16) circle [radius=\dt];
\draw [fill] (03,14) circle [radius=\dt]; 
\draw [fill] (03,12) circle [radius=\dt]; 
\draw [fill] (00,10) circle [radius=\dt]; \draw [fill] (06,10) circle [radius=\dt];
\draw [fill] (03,08) circle [radius=\dt]; \draw [fill] (09,08) circle [radius=\dt];
\draw [fill] (06,06) circle [radius=\dt];
\draw [fill] (06,04) circle [radius=\dt];
\draw [fill] (06,02) circle [radius=\dt];
\draw [fill] (06,00) circle [radius=\dt];
\draw 
(06,00) -- (06,06) -- (09,08) -- (03,12) -- (03,18) 
(06,06) -- (03,08) -- (06,10) (03,08) -- (00,10) -- (03,12)
;
\node [right] at (03.2,18.1) {$\bfI$};
\node [right] at (03.2,16) {$\bfIsub$};
\node [right] at (03.2,14) {$\bfPsub$};
\node [right] at (03.2,12.2) {$\bfFsub \vee \bfEdual$};
\node [left] at (00,10) {$\bfFsub$}; \node [right] at (06.2,10.1) {$\bfBzero$};
\node [left] at (2.8,07.9) {$\bfE$}; \node [right] at (09.1,08.15) {$\bfEdual$};
\node [right] at (06.2,05.8) {$\bfRq\{xy\}$};
\node [right] at (06.2,03.95) {$\bfRq\{x\}$};
\node [right] at (06.2,01.95) {$\bfRq\{1\}$};
\node [right] at (06.2,00) {$\bftrivial$};
\end{tikzpicture}
\end{center} \caption{The lattice $\frakL(\bfI)$}
\label{F: I}
\end{figure}

\subsection{The varieties \texorpdfstring{$\bfK$}{K} and \texorpdfstring{$\bfKdual$}{K-}} \label{subsec: var K}

Let~$\bfK$ denote the variety generated by the monoid
\begin{align*}
\monK & = \langle a,b,e,\, 1 \,|\, a^3=a^2,\; be=b,\; e^2=e,\; a^2b=ae=eab=eb=0,\; aba^2=aba \rangle \\
& = \{ 0,\, a,\, b,\, e,\, a^2,\, ab,\, ba,\, ea,\, aba,\, ba^2,\, ea^2,\, 1\}.
\end{align*}
The variety~$\bfK$ is not self-dual because $\monK \models \eqref{id: xyxx=xyx}$ and $\monK \not\models \eqref{id: xxyx=xyx}$.

\begin{proposition}[{Gusev and Sapir \cite[Propositions~5.1 and~6.5]{GusSap22}}] \label{P: K}
\begin{enumerate}[\rm(i)] 
\item The variety $\bfK$ is a non-finitely based\textup, finitely generated almost Cross variety.
\item The lattice $\frakL(\bfK)$ is given in Figure~\ref{F: K}\textup; the varieties in this lattice that have not been introduced are
\begin{align*}
\bfFsub \vee \bfAzero & = \var\{ \eqref{id: xyxx=xyx},\; xyhxty \approx yxhxty,\; xhytxy \approx xhytyx,\; xh_1xh_2xh_3x \approx xh_1xh_2h_3x \} \\
\text{and } \ \bfKsub & = \var\{ \eqref{id: xyxx=xyx},\; xyhxty \approx yxhxty,\; x^2y^2 \approx y^2x^2,\; xh_1xh_2xh_3x \approx xh_1xh_2h_3x \}. 
\end{align*}
\end{enumerate}
\end{proposition}

\begin{figure}[ht]
\begin{center}
\begin{tikzpicture}[scale=0.4]
\draw [fill] (06,16) circle [radius=\dt];
\draw [fill] (06,14) circle [radius=\dt];
\draw [fill] (04.5,13) circle [radius=\dt]; 
\draw [fill] (03,12) circle [radius=\dt]; \draw [fill] (09,12) circle [radius=\dt]; 
\draw [fill] (00,10) circle [radius=\dt]; \draw [fill] (06,10) circle [radius=\dt];
\draw [fill] (03,08) circle [radius=\dt]; \draw [fill] (09,08) circle [radius=\dt];
\draw [fill] (06,06) circle [radius=\dt];
\draw [fill] (06,04) circle [radius=\dt];
\draw [fill] (06,02) circle [radius=\dt];
\draw [fill] (06,00) circle [radius=\dt];
\draw 
(06,00) -- (06,06) -- (09,08) -- (03,12) -- (06,14) -- (06,16) 
(06,06) -- (03,08) -- (09,12) -- (06,14) 
(03,08) -- (00,10) -- (03,12)
;
\node [right] at (06.1,16.1) {$\bfK$};
\node [right] at (06.1,14.1) {$\bfFsub \vee \bfAzero$};
\node [left] at (04.4,13.3) {$\bfKsub$};
\node [left] at (2.9,12.2) {$\bfFsub \vee \bfEdual$}; \node [right] at (09.1,12) {$\bfAzero$};
\node [left] at (00,10) {$\bfFsub$}; \node [right] at (06.3,09.9) {$\bfBzero$};
\node [left] at (2.8,07.9) {$\bfE$}; \node [right] at (09,08.15) {$\bfEdual$};
\node [right] at (06.2,05.8) {$\bfRq\{xy\}$};
\node [right] at (06.2,03.95) {$\bfRq\{x\}$};
\node [right] at (06.2,01.95) {$\bfRq\{1\}$};
\node [right] at (06.2,00) {$\bftrivial$};
\end{tikzpicture}
\end{center} \caption{The lattice $\frakL(\bfK)$}
\label{F: K}
\end{figure}

\begin{remark}
The equivalence relation~$\sim$ on~$\monK$ with non-singleton classes $\{ba,ba^2\}$ and $\{ea,ea^2\}$ is a congruence on~$\monK$.
The quotient \[ {\monK/\sim} \, = \{ 0,\, a,\, b,\, e,\, a^2,\, ab,\, ba,\, ea,\, aba,\, 1\} \] does not satisfy the identity $xhytxy \approx xhytyx$ of $\bfFsub \vee \bfAzero$ because $abe1ae \neq abe1ea$.
Since $\bfFsub \vee \bfAzero$ is the unique maximal subvariety of~$\bfK$, it follows that~$\bfK$ is also generated by~$\monK/\sim$.
\end{remark}

\begin{lemma} \label{L: excl I K}
Let $\bfV$ be any variety in the interval $[\bfFsub \vee \bfEdual,\, \bfJ_2\{\eqref{id: xyxx=xyx}\}]$.
Suppose that $\bfI,\bfK \nsubseteq \bfV$.
Then $\bfV \subseteq \bfO \cap \bfJ_2$.
\end{lemma}

\begin{proof}
The assumption $\bfI,\bfK \nsubseteq \bfV$ implies that~$\bfV$ satisfies the identities
\begin{align}
xhxyxty & \approx xhyxty, \label{id: excl I} \\
xhy^2x & \approx xhy(yx)^2; \label{id: excl K}
\end{align}
see Gusev \cite[Lemma~4.1]{Gus25ijac} and Gusev and Sapir \cite[Lemma~4.3]{GusSap22}.
Since \[ \bfV \models xhytxyx \stackrel{\eqref{id: excl I}} {\approx} xhyt(yx)^2 \stackrel{\eqref{id: xyxx=xyx}}{\approx} xhyty(yx)^2 \stackrel{\eqref{id: excl K}}{\approx} xhyty^2x \stackrel{\eqref{id: xyxx=xyx}}{\approx} xhytyx \vdash xhytxyx \approx xhytyx, \] the variety satisfies the identities~\eqref{id: O}, so that $\bfV \subseteq \bfO \cap \bfJ_2$.
\end{proof}

\subsection{The varieties \texorpdfstring{$\bfJackson_1$}{Y1} and \texorpdfstring{$\bfJackson_2$}{Y2}} \label{subsec: var Y1 Y2}


\begin{proposition}[Jackson~\cite{Jac05}] \label{P: Jackson}
\begin{enumerate}[\rm(i)] 
\item The varieties \[ \bfJackson_1 = \bfRq\{xhxyty\} \quad \text{and} \quad \bfJackson_2 = \bfJackson_3 \vee \overleftarrow{\bfJackson}\!{}_3, \] where \begin{align*} \bfJackson_3 & = \bfRq\{xhytxy\} = \var\bigg\{ \begin{array}{l} x^2h \approx hx^2, \; xhxtx \approx x^2ht, \; xyhxty \approx yxhxty, \\ xhxyty \approx xhyxty, \; xyhxy \approx xyhyx \end{array} \bigg\} \\ \text{and } \ \overleftarrow{\bfJackson}\!{}_3 & = \bfRq\{xyhxty\} = \var\bigg\{ \begin{array}{l} x^2h \approx hx^2, \; xhxtx \approx x^2ht, \; xyhxy \approx yxhxy, \\ xhxyty \approx xhyxty, \; xhytxy \approx xhytyx \end{array} \bigg\}, \end{align*} are non-finitely based\textup, finitely generated almost Cross varieties.
\item The lattices $\frakL(\bfJackson_1)$ and $\frakL(\bfJackson_2)$ are given in Figure~\ref{F: Jackson}.
\end{enumerate}
\end{proposition}

\begin{figure}[ht]
\begin{center}
\begin{tikzpicture}[scale=0.4]
\draw [fill] (00,10) circle [radius=\dt];
\draw [fill] (00,08) circle [radius=\dt];
\draw [fill] (00,06) circle [radius=\dt];
\draw [fill] (00,04) circle [radius=\dt];
\draw [fill] (00,02) circle [radius=\dt];
\draw [fill] (00,00) circle [radius=\dt];
\draw (00,00) -- (00,10);
\node [above] at (0.15,10) {$\bfJackson_1$};
\node [right] at (0.2,07.95) {$\bfRq\{aba\}$};
\node [right] at (0.2,05.95) {$\bfRq\{ab\}$};
\node [right] at (0.2,03.95) {$\bfRq\{a\}$};
\node [right] at (0.2,01.95) {$\bfRq\{1\}$};
\node [right] at (0.2,00) {$\bftrivial$};
\draw [fill] (12,12) circle [radius=\dt];
\draw [fill] (09,10) circle [radius=\dt]; \draw [fill] (15,10) circle [radius=\dt];
\draw [fill] (12,08) circle [radius=\dt];
\draw [fill] (12,06) circle [radius=\dt];
\draw [fill] (12,04) circle [radius=\dt];
\draw [fill] (12,02) circle [radius=\dt];
\draw [fill] (12,00) circle [radius=\dt];
\draw (12,00) -- (12,08) -- (09,10) -- (12,12) -- (15,10) -- (12,08);
\node [above] at (12.15,12) {$\bfJackson_2$};
\node [left] at (8.95,10) {$\bfJackson_3$}; \node [right] at (15,10.15) {$\overleftarrow{\bfJackson}\!{}_3$};
\node [right] at (12.2,07.8) {$\bfRq\{aba\}$};
\node [right] at (12.2,05.95) {$\bfRq\{ab\}$};
\node [right] at (12.2,03.95) {$\bfRq\{a\}$};
\node [right] at (12.2,01.95) {$\bfRq\{1\}$};
\node [right] at (12.2,00) {$\bftrivial$};
\end{tikzpicture}
\end{center} \caption{The lattices $\frakL(\bfJackson_1)$ and $\frakL(\bfJackson_2)$}
\label{F: Jackson}
\end{figure}

\begin{lemma} \label{L: xyx isoterm}
Let~$\bfV$ be any subvariety of~$\bbJ$ such that $\monRq\{xhx\} \in \bfV$.
Then~$\bfV$ is Cross if and only if $\bfL, \bfP, \bfPdual, \bfJackson_1, \bfJackson_2 \nsubseteq \bfV$.
\end{lemma}

\begin{proof}
The varieties $\bfL, \bfP, \bfPdual, \bfJackson_1, \bfJackson_2$ are almost Cross by Propositions~\ref{P: L}, \ref{P: P}, and~\ref{P: Jackson}.
Hence if~$\bfV$ is Cross, then clearly $\bfL, \bfP, \bfPdual, \bfJackson_1, \bfJackson_2 \nsubseteq \bfV$.

Conversely, suppose $\bfL, \bfP, \bfPdual, \bfJackson_1, \bfJackson_2 \nsubseteq \bfV$.
Then since $\bfRq\{xhx\} \subseteq \bfV$,\, $\bfRq\{ xhxyty \} \nsubseteq \bfV$, and $\bfRq\{ xhytxy \} \vee \bfRq\{ xyhxty\} \nsubseteq \bfV$,
by Lemma~\ref{L: isoterm}, one of the following holds:
\begin{enumerate}[\ (a)]
\item $xhx$ is an isoterm for~$\bfV$ and $xhxyty$ and $xhytxy$ are not isoterms for~$\bfV$;
\item $xhx$ is an isoterm for~$\bfV$ and $xhxyty$ and $xyhxty$ are not isoterms for~$\bfV$.
\end{enumerate}
It follows that~$\bfV$ is contained in the variety $\mathbf{D} = \var\{ xhxyty \approx xhyxty,\; xhytxy \approx xhytyx \}$ or its dual~$\overleftarrow{\mathbf{D}}$ \cite[Lemma~2.2]{Lee12}.
But~$\bfL$ and~$\bfP$ are the only almost Cross subvarieties of~$\mathbf{D}$, while~$\bfL$ and~$\bfPdual$ are the only almost Cross subvarieties of~$\overleftarrow{\mathbf{D}}$ \cite[Proposition~4.1]{Gus25ijac}.
Consequently, $\bfV$ does not contain any almost Cross subvarieties and so is Cross.
\end{proof}

\subsection{The variety \texorpdfstring{$\bfZL$}{Z}} \label{subsec: var ZL}

Recall from Remark~\ref{R: Hm} that variety $\bfHsub = \bfH\{ \text{\ref{id: Hsub}}_3 \}$ is generated by the monoid~$\monHsub$.
Although the varieties~$\bfHsub$ and~$\bfHsubdual$ are finitely based, their join \[ \bfZL = \bfHsub \vee \bfHsubdual \] is non-finitely based~\cite{ZhaLuo16}.
In fact, Zhang and Luo~\cite{ZhaLuo19} have shown that~$\bfZL$ is almost Cross; however, their result has remained unpublished.
An independent, shorter proof of this result is given in the present subsection.

\begin{lemma} \label{L: excl H3}
Let~$\bfV$ be any variety in the interval $[\bfQ,\, \bfJ_2\{\eqref{id: xxyx=xyx},\eqref{id: xyxx=xyx} \}]$.
Then
\begin{enumerate}[\rm(i)] 
\item $\monHsub \notin \bfV$ implies that $\bfV \subseteq \bfHdual$\textup;
\item $\monHsubdual \notin \bfV$ implies that $\bfV \subseteq \bfH$.
\end{enumerate}
\end{lemma}

\begin{proof}
(i) Suppose that $\monHsub \notin \bfV$.
Then since $\bfH = \bfJ_2 \{ \eqref{id: xxyx=xyx},\, \eqref{id: xyxx=xyx},\, xhxy^2x \approx xhy^2x \}$, it suffices to show that~$\bfV$ satisfies the identity \begin{equation} xy^2xtx \approx xy^2tx. \label{id: basis Hdual} \end{equation}
Now since the monoid~$\monHsub$ is embeddable in the monoid $M_\gamma(\mathtt{a^+b^+ta^+})$ from O.B.\ Sapir \cite[Remark~4.4(ii)]{SapO21}, the assumption $\monHsub \notin \bfV$ implies that $M_\gamma(\mathtt{a^+b^+ta^+}) \notin \bfV$ \cite[Corollary~3.6]{SapO21}; and since $\bfQ \subseteq \bfV$ by assumption, it follows from Lemma~\ref{L: id Q} that~$\bfV$ satisfies an identity $xy^2tx \approx \bfw tx$ for some word $\bfw \in \{x,y\}^*$ containing~$yx$ as a factor.
Multiplying both sides of this identity on the left by~$x$ and making the substitution $t \mapsto yt$ result in the identity $x^2y^3tx \approx x\bfw ytx$ of~$\bfV$.
Then
\begin{align*}
\bfV & \models \{ \text{\ref{id: aperiodic}}_2,\, \text{\ref{id: eventually}}_2,\, \eqref{id: xxyx=xyx},\, \eqref{id: xyxx=xyx},\, x^2y^3tx \approx x\bfw ytx\} \\
& \vdash \{ \text{\ref{id: aperiodic}}_2,\, \text{\ref{id: eventually}}_2,\, \eqref{id: xxyx=xyx},\, \eqref{id: xyxx=xyx},\, xy^2tx \approx (xy)^2tx\} \\
& \vdash xy^2tx \approx (xy)^2tx \stackrel{\text{\ref{id: eventually}}_2}{\approx} (yx)^2tx \stackrel{\eqref{id: xxyx=xyx}}{\approx} (yx)^2xtx \stackrel{\text{\ref{id: eventually}}_2}{\approx} (xy)^2xtx \approx xy^2xtx \\ 
& \vdash \eqref{id: basis Hdual}. 
\end{align*}

(ii) This is symmetrical to part~(i). 
\end{proof}

\begin{proposition}[{Zhang and Luo~\cite{ZhaLuo19}}] \label{P: ZL}
\begin{enumerate}[\rm(i)] 
\item The variety~$\bfZL$ is a non-finitely based\textup, finitely generated almost Cross variety.
\item The lattice $\frakL(\bfZL)$ is given in Figure~\ref{F: ZL}.
\end{enumerate}
\end{proposition}

\begin{proof}
The variety~$\bfZL$ is non-finitely based by Zhang and Luo~\cite{ZhaLuo16}; see also O.B.\ Sapir \cite[Corollary~5.5]{SapO21}.
Since~$\bfZL$ is generated by the monoids~$\monHsub$ and~$\monHsubdual$, it is routinely checked that
\begin{enumerate}[\ (a)]
\item $\bfZL \subseteq \bfJ_2\{\eqref{id: xxyx=xyx},\eqref{id: xyxx=xyx},\, \text{\ref{id: Hsub}}_3\}$ and $\bfZL \subseteq \overleftarrow{\;\bfJ_2\{\eqref{id: xxyx=xyx},\eqref{id: xyxx=xyx},\, \text{\ref{id: Hsub}}_3\}}$.
\end{enumerate}

Let~$\bfV$ be any proper subvariety of~$\bfZL$, so that either $\monHsub \notin \bfV$ or $\monHsubdual \notin \bfV$.
If $\monQ \notin \bfV$, then $\bfV \subseteq \bfAzero$ by Corollary~\ref{C: exclusion A0 Q}(ii), so that $\bfV \subseteq \bfHsub$ by Figure~\ref{F: H}.
If $\monQ \in \bfV$, then by Lemma~\ref{L: excl H3}, either $\bfV \subseteq \bfHdual$ or $\bfV \subseteq \bfH$, so that by~(a), either \[ \bfV \subseteq \bfHdual \cap \overleftarrow{\;\bfJ_2\{\eqref{id: xxyx=xyx},\eqref{id: xyxx=xyx},\, \text{\ref{id: Hsub}}_3\}} = \bfHsubdual \quad \text{or} \quad \bfV \subseteq \bfH \cap \bfJ_2\{\eqref{id: xxyx=xyx},\eqref{id: xyxx=xyx},\, \text{\ref{id: Hsub}}_3\} = \bfHsub. \]
Therefore,
\begin{enumerate}[\ (a)]
\item[(b)] $\bfHsub$ and~$\bfHsubdual$ are the only maximal subvarieties of~$\bfZL$.
\end{enumerate}

By Lemma~\ref{L: H2} and Figure~\ref{F: H}, the variety $\bfH_2 = \bfAzero \vee \bfQ$ is the unique maximal subvariety of~$\bfHsub$.
But since $\bfH_2$ is self-dual (see Remark~\ref{R: Hm}), it is also the unique maximal subvariety of~$\bfHsubdual$.
In other words,
\begin{enumerate}[\ (a)]
\item[(c)] $\bfH_2$ is the unique maximal subvariety of~$\bfHsub$ and of~$\bfHsubdual$.
\end{enumerate}
The description of $\frakL(\bfZL)$ in Figure~\ref{F: ZL} now follows from~(b), (c), and the description of $\frakL(\bfH_2)$ in Figure~\ref{F: H}.
Since the varieties~$\bfHsub$ and~$\bfHsubdual$ are Cross by Proposition~\ref{P: H}, the non-finitely based variety~$\bfZL$ is almost Cross.
\end{proof}

\begin{figure}[ht]
\begin{center}
\begin{tikzpicture}[scale=0.4]
\draw [fill] (03,18) circle [radius=\dt];
\draw [fill] (00,16) circle [radius=\dt]; \draw [fill] (06,16) circle [radius=\dt]; 
\draw [fill] (03,14) circle [radius=\dt];
\draw [fill] (00,12) circle [radius=\dt]; \draw [fill] (06,12) circle [radius=\dt]; 
\draw [fill] (03,10) circle [radius=\dt];
\draw [fill] (00,08) circle [radius=\dt]; \draw [fill] (06,08) circle [radius=\dt];
\draw [fill] (03,06) circle [radius=\dt];
\draw [fill] (03,04) circle [radius=\dt];
\draw [fill] (03,02) circle [radius=\dt];
\draw [fill] (03,00) circle [radius=\dt];
\draw (03,00) -- (03,06) -- (00,08) -- (06,12) -- (03,14) -- (00,16) -- (03,18) -- (06,16) -- (00,12) -- (06,08) -- (03,06);
\node [above] at (03,18.15) {$\bfZL$};
\node [left] at (-0.15,15.95) {$\bfH_3$}; \node [right] at (06.1,16.1) {$\overleftarrow{\bfH}\!{}_3$};
\node [right] at (03.3,13.95) {$\bfH_2$};
\node [left] at (-0.2,12) {$\bfAzero$}; \node [right] at (06.2,12) {$\bfQ$};
\node [right] at (03.3,9.95) {$\bfBzero$};
\node [left] at (-0.2,08) {$\bfE$}; \node [right] at (06.1,08.15) {$\bfEdual$};
\node [right] at (03.2,05.8) {$\bfRq\{xy\}$};
\node [right] at (03.2,03.95) {$\bfRq\{x\}$};
\node [right] at (03.2,01.95) {$\bfRq\{1\}$};
\node [right] at (03.2,00) {$\bftrivial$};
\end{tikzpicture}
\end{center} \caption{The lattice $\frakL(\bfZL)$}
\label{F: ZL}
\end{figure}

\begin{remark}
The reference items that cited Zhang and Luo~\cite{ZhaLuo19} are the three articles by Gusev~\cite{Gus25sf}, Gusev and Sapir~\cite{GusSap22}, and O.B.\@~Sapir \cite{SapO21}; the precise results from these three articles that are required in the present article are
\begin{enumerate}[(i)] 
\item Gusev and Sapir \cite[Lemma~2.7]{GusSap22},
\item Gusev and Sapir \cite[proof of Proposition~3.1]{GusSap22}, 
\item Gusev \cite[Lemma~6.7]{Gus25sf}, and
\item O.B.\ Sapir \cite[Corollary~3.6 and Remark~4.4(ii)]{SapO21}.
\end{enumerate}
Specifically, (i)--(iv) are used to establish Lemmas~\ref{L: comm CR}(ii), \ref{L: O sub}(i), \ref{L: Oc sub}, and~\ref{L: excl H3}, respectively.
But one can routinely check that (i)--(iv) were established without using any results from Zhang and Luo~\cite{ZhaLuo19}.
Consequently, the  proof of Proposition~\ref{P: ZL} is independent of Zhang and Luo~\cite{ZhaLuo19}.
\end{remark}

\section{Characterization of Cross varieties} \label{sec: characterization}

\begin{theorem} \label{T: main}
A variety of {\Jtrivial} monoids is Cross if and only if it excludes every one of the following almost Cross varieties\textup:
\begin{equation}
\bfF, \;\;\; \bfFdual, \;\;\; \bfH, \;\;\; \bfHdual, \;\;\; \bfI, \;\;\; \bfIdual, \;\;\; \bfK, \;\;\; \bfKdual, \;\;\; \bfL, \;\;\; \bfP, \;\;\; \bfPdual, \;\;\; \bfJackson_1, \;\;\; \bfJackson_2, \;\;\; \bfZL.
\label{D: 14 almost}
\end{equation}
Consequently\textup, these 14 varieties exhaust all almost Cross varieties of {\Jtrivial} monoids.
\end{theorem}

\begin{proof} 
The varieties from~\eqref{D: 14 almost} were shown to be almost Cross in Sections~\ref{sec: almost nfg} and~\ref{sec: almost fg}, so they are excluded from any Cross variety.
Conversely, suppose that~$\bfV$ is a Cross variety of {\Jtrivial} monoids, so that \begin{enumerate}[\ (a)] \item $\bfV$ contains none of the varieties from~\eqref{D: 14 almost}. \end{enumerate}
If~$\bfV$ is either commutative or completely regular, then it is Cross by Lemma~\ref{L: comm CR}.
Therefore, suppose that~$\bfV$ is neither commutative nor completely regular, whence by Lemma~\ref{L: subvarieties of J}(ii), it is a subvariety of~$\bfJ_n$ for some $n \geq 2$.

If $\monRq\{xhx\} \in \bfV$, then~$\bfV$ is Cross by Lemma~\ref{L: xyx isoterm}.
Hence assume that $\monRq\{xhx\} \notin \bfV$, so that by Lemma~\ref{L: xyx not isoterm}(i), $\bfV$ is a subvariety of one of the following varieties: \[ \bfJ_2\{\eqref{id: xxyx=xyx}\}, \quad \bfJ_2\{\eqref{id: xyxx=xyx}\}, \quad \bfJ_n\{xhx \approx x^2h\}, \quad \bfJ_n\{xhx \approx hx^2\}, \quad n \geq 2. \]
Since the varieties $\bfJ_n\{xhx\approx x^2h\}$ and $\bfJ_n\{xyx \approx hx^2\}$ are Cross by Lemma~\ref{L: xyx not isoterm}(ii), by symmetry, it suffices to assume that \begin{enumerate}[\ (a)] \item[(b)] $\bfV \subseteq \bfJ_2\{\eqref{id: xyxx=xyx}\}$. \end{enumerate}
If $\monEdual \notin \bfV$, then since $\bfF \nsubseteq \bfV$ by~(a), it follows from~(b) and Lemma~\ref{L: excl Fsub Edual}(ii) that~$\bfV$ is Cross.
Hence assume that \begin{enumerate}[\ (a)] \item[(c)] $\monEdual \in \bfV$. \end{enumerate}
There are two cases to consider.

\medskip

\noindent\textsc{Case~1}: $\monFsub \notin \bfV$.
Then $\bfV \models \eqref{id: xxyx=xyx}$ by Lemma~\ref{L: excl Fsub Edual}(i), so that $\bfV \subseteq \bfJ_2\{\eqref{id: xxyx=xyx},\eqref{id: xyxx=xyx}\}$ by~(b).
If $\monQ \notin \bfV$, then $\bfV \subseteq \bfAzero$ by Corollary~\ref{C: exclusion A0 Q}(ii), so that~$\bfV$ is Cross by Proposition~\ref{P: A0 vee Q}.
If $\monQ \in \bfV$, so that $\bfV \in [\bfQ,\, \bfJ_2\{\eqref{id: xxyx=xyx},\eqref{id: xyxx=xyx} \}]$, then since $\monHsub, \monHsubdual \in \bfZL \nsubseteq \bfV$ by~(a), it follows from Lemma~\ref{L: excl H3} that either $\bfV \subseteq \bfHdual$ or $\bfV \subseteq \bfH$, whence~$\bfV$ is Cross due to $\bfV \neq \bfH, \bfHdual$ and Proposition~\ref{P: H}.

\medskip

\noindent\textsc{Case~2}: $\monFsub \in \bfV$.
Then $\bfV \in [\bfFsub \vee \bfEdual,\, \bfJ_2\{\eqref{id: xyxx=xyx}\}]$ by~(b) and~(c).
Since $\bfI,\bfK \nsubseteq \bfV$ by~(a), the inclusion $\bfV \subseteq \bfO$ holds by Lemma~\ref{L: excl I K}, so that $\bfV \subseteq \bfO \cap \bfJ_2\{\eqref{id: xyxx=xyx}\} = \bfO \cap \bfJ_2$.
If $\monQ \notin \bfV$, then it follows from Lemma~\ref{L: exclusion A0 Q B0}(ii) that $\bfV \models x^2hxtx \approx x^2htx$, whence~$\bfV$ is Cross by Corollary~\ref{C: O cap J2}(i).
If $\monQ \in \bfV$, so that $[\bfFsub \vee \bfQ,\, \bfJ_2\{\eqref{id: xyxx=xyx}\}]$, then since $\bfP \nsubseteq \bfV$ by~(a), it follows from Lemma~\ref{L: excl P} that $\bfV \models \text{\ref{id: Hsub}}_m$ for some $m \geq 2$, whence~$\bfV$ is Cross by Corollary~\ref{C: O cap J2}(ii).
\end{proof}

%
%
%

\end{document}